\documentclass[reqno,tbtags]{amsproc}
\usepackage{amssymb}
\usepackage{amsmath}
\usepackage{ecltree}
\usepackage{epic}
\usepackage{epsfig}
\usepackage{eufrak}
\usepackage{euscript}
\usepackage{fancybox}
\usepackage{graphicx}
\usepackage{multicol}
\usepackage{subfigure}
\usepackage{wrapfig}

\newtheorem{theorem}{Theorem}[section]

\theoremstyle{definition}
\newtheorem{definition}{Definition}[section]
\newtheorem{remark}{Remark}[section]

\newtheorem{example}{Example}[section]

\def\AASF{\emph{Ann. Academi{\ae} Scientiarum Fennic{\ae}\/}}

\def\AIHP{\emph{Ann. Inst. H. Po\-in\-ca\-re\/}}
\def\AP{\emph{Ann. Probab.\/}}

\def\AMASH{\emph{Acta Math. Acad. Sci. Hungar.}}
\def\AMS{\emph{Ann. Math. Stat.\/}}
\def\AAP{\emph{Ad\-van\-ces in Ap\-pli\-ed Pro\-ba\-bi\-li\-ty\/}}
\def\ASENS{\emph{Ann. Sci. {\'E}cole Norm. Sup.\/}}
\def\BAMS{\emph{Bull. Amer. Math. Soc.}\/} 
\def\BSMF{\emph{Bull. Soc. Math. de France\/}}
\def\CACM{\emph{Communications of the Association for Computing Machinery (ACM)\/}}
\def\ConM{\emph{Contemporary Mathematics\/}}

\def\IANSSSR{\emph{Izv. Akad. Nauk SSSR\/}}

\def\JAP{\emph{Journal of Applied Probability\/}}
\def\JCAM{\emph{Journal of Computational and Applied Mathematics\/}}

\def\JSM{\emph{Journal of Soviet Math.\/}}
\def\JMAA{\emph{J. Math. Anal. Appl.\/}}
\def\LMJ{\emph{Lith. Math. J.\/}}
\def\MCMA{\emph{Monte Carlo Methods and Appl.\/}}
\def\MS{\emph{Management Science\/}}

\def\PCPS{\emph{Proc. Cam\-bridge Phil. Soc.\/}}
\def\PJM{\emph{Pacific J. Math.\/}}
\def\PMS{\emph{Pro\-ba\-bility and Mathematical Statistics\/}}
\def\PTRF{\emph{Probability Theory and Related Fields\/}}
\def\ReM{\emph{Mat. Sbornik, N. Ser.\/}}
\def\RHSA{\emph{Rev. Histoire Sci. Appl.\/}}
\def\RMUI{\emph{Revue Math. de l'Union Interbalkanique\/}}
\def\SAJ{\emph{Scan\-di\-na\-vian Ac\-tua\-rial Journal\/}}

\def\SMD{\emph{Soviet Math. Dokl.\/}}
\def\SPA{\emph{Sto\-chas\-tic Pro\-ces\-ses and their Ap\-pli\-ca\-ti\-ons\/}}
\def\TAMS{\emph{Trans. Amer. Math. Soc.\/}}
\def\TPA{\emph{Theory Probab. Appl.}}
\def\ZW{\emph{Z. Wahr\-schein\-lich\-keits\-the\-orie Verw. Geb.}}
\numberwithin{equation}{section}
\def\eqD{\overset{d}{=}}
\def\wrt{w.r.t.\hskip 4pt}
\def\cdf{c.d.f.\hskip 4pt}
\def\pdf{p.d.f.\hskip 4pt}
\def\cpdf{p.d.f.}
\def\as{a.s.\hskip 4pt}
\def\iid{i.i.d.\hskip 4pt}
\def\zzz{--}
\def\zzzx{\hskip 4pt}
\def\paramX{\varrho}
\def\paramY{\rho}
\def\paramT{\delta}
\def\CLT{\textsl{CLT}\;}
\def\cCLT{\textsl{CLT}}
\def\WLLN{\textsl{WLLN}\;}
\def\cWLLN{\textsl{WLLN}}
\def\SLLN{\textsl{SLLN}\;}
\def\cSLLN{\textsl{SLLN}}

\def\cLIL{\textsl{LIL}}
\def\mRI{$\mathcal{R\hskip-.5pt/I}$}
\def\mII{$\mathcal{I\hskip-.5pt/I}$}
\def\mRD{$\mathcal{R\hskip-.5pt/D}$}
\def\mID{$\mathcal{I\hskip-.5pt/D}$}
\def\bmRI{$\boldsymbol{\mathcal{R\hskip-.5pt/I}}$}
\def\bmII{$\boldsymbol{\mathcal{I\hskip-.5pt/I}}$}
\def\bmRD{$\boldsymbol{\mathcal{R\hskip-.5pt/D}}$}
\def\bmID{$\boldsymbol{\mathcal{I\hskip-.5pt/D}}$}

\renewcommand{\P}{\mathsf{P}}
\def\toP{\overset{\scriptscriptstyle{\P}}{\rightarrow}}

\newcommand{\Space}[1]{\mathit{#1}}
\newcommand{\Field}[1]{\mathcal{#1}}
\newcommand{\bKovSN}[2]{\mathcal{B}_{\Sum{S}{}\infLS{}}^{#1}{\,#2}}
\newcommand{\AssSum}[2]{{\olss{#1}}_{#2}}
\newcommand{\infLSAss}[1]{{\bar{\mathcal{N}}}_{#1}}
\newcommand{\AssBlock}[3]{{\bar{\mathcal#1}}_{#2}^{#3}}

\newcommand{\ols}[1]{\mskip0.2\thinmuskip\overline{\mskip-0.3\medmuskip{#1}\mskip0.3\medmuskip}\mskip0.5\thinmuskip} 
\newcommand{\LaddAss}[1]{\,\ols{\mathbb{\ell}}_{#1}}
\newcommand{\olss}[1]{\mskip0.5\thinmuskip\overline{\mskip-0.5\medmuskip{#1}\mskip-0.6\medmuskip}\mskip0.5\thinmuskip} 
\newcommand{\Ass}[2]{\olss{#1}_{#2}}
\newcommand{\olssX}[1]{\mskip1.2\thinmuskip\overline{\mskip-1.2\medmuskip{#1}\mskip-0.5\medmuskip}\mskip0.5\thinmuskip} 
\newcommand{\AssX}[1]{\olssX{X}_{#1}}

\newcommand{\olssT}[1]{\mskip0.5\thinmuskip\overline{\mskip-0.5\medmuskip{#1}\mskip-0.3\medmuskip}\mskip0.3\thinmuskip} 
\newcommand{\AssT}[1]{\olssT{T}_{#1}}

\newcommand{\AInt}[2]{{\mathcal{#1}}_{#2}}
\newcommand{\muIG}{\mu}
\newcommand{\lambdaIG}{\lambda}
\newcommand{\HmuIG}{\hat{\mu}}
\newcommand{\LunKonst}{\mathbb{C}}
\newcommand{\1}{\boldsymbol{\mathsf 1}}
\newcommand{\LunAdjust}{\varkappa}
\newcommand{\BinfLS}[1]{\mathcal{N}_{#1}}

\newcommand{\Garb}[1]{\EuScript{G}_{\,#1}}
\newcommand{\const}[3]{{#1}_{#2}^{#3}}
\newcommand{\centrBlock}[3]{\hat{\mathcal#1}_{#2}^{\hskip 1pt#3}}
\newcommand{\moplus}{m_{_{\scriptscriptstyle\vartriangle}}}
\newcommand{\mominus}{m_{_{\scriptscriptstyle\triangledown}}}
\newcommand{\Doplus}{D_{{\hskip -1pt\scriptscriptstyle\vartriangle}}}
\newcommand{\Dominus}{D_{{\hskip -1pt\scriptscriptstyle\triangledown}}}
\newcommand{\Ladd}[1]{\mathbb{\ell}_{#1}}
\newcommand{\RemTerm}[1]{\mathcal{R}_{#1}}
\newcommand{\SeqApprox}[2]{\mathcal{I}_{#1}^{\,#2}}
\newcommand{\HNode}[4]{\lambda_{#1,#2}^{#4}{(#3)}}
\newcommand{\Approx}[1]{\mathcal{I}_{#1}}
\newcommand{\bFKovSN}[2]{\mathcal{\rho}_{\Sum{S}{}\!\infLS{}}^{#1}{\,#2}}
\newcommand{\Renyi}[1]{R_{#1}}
\newcommand{\stand}[2]{\tilde{#1}_{#2}}
\newcommand{\parMu}[2]{\mu_{#1}^{#2}}
\newcommand{\parSigma}[2]{\sigma_{#1}^{#2}}
\newcommand{\xEta}[1]{\eta_{\hskip 1pt #1}}
\newcommand{\MeanS}[2]{M_{\Sum{S}{}}^{#1}{\,#2}}
\newcommand{\bMeanS}[2]{\mathcal{M}_{\Sum{S}{}}^{#1}{\hskip 0.6pt#2}}
\newcommand{\bMeanN}[2]{\mathcal{M}_{\infLS{}}^{#1}{\,#2}}
\newcommand{\StDevS}[2]{D_{\Sum{S}{}}^{#1}{\,#2}}

\newcommand{\bStDevS}[2]{\mathcal{D}_{\Sum{S}{}}^{\,#1}{\,#2}}
\newcommand{\bStDevN}[2]{\mathcal{D}_{\infLS{}}^{\,#1}{\,#2}}
\newcommand{\bKompl}[2]{\mathcal{K}_{(#1,#2)}}
\newcommand{\cdF}[1]{F_{#1}}
\newcommand{\pdF}[1]{f_{#1}}
\newcommand\Step[2]{\smallskip{\sl Step\hskip 4pt{#1}}\textbf{\sl:\hskip 4pt#2.}\hskip 4pt}
\newcommand\Steps[1]{\smallskip{\sl Steps\hskip 4pt{#1}}\hskip 4pt}
\newcommand{\RenH}{\mathsf{H}}
\newcommand{\RenU}{\mathsf{U}}
\newcommand{\Polynom}[2]{\mathsf{#1}_{\,#2}}
\newcommand{\simb}[1]{#1}
\newcommand{\kov}[2]{\varkappa_{#1}^{#2}}
\newcommand{\hkov}[2]{\varkappa_{#1}^{\,(#2)}}
\newcommand{\Ugauss}[2]{\varphi_{\left({#1},{#2}\right)}}
\newcommand{\UGauss}[2]{\varPhi_{\left({#1},{#2}\right)}}
\newcommand{\Kompl}[2]{K_{(#1,#2)}}
\newcommand{\Herm}[1]{\mathsf{H}_{#1}}
\newcommand{\Value}[2]{{#1}_{#2}}
\newcommand{\rmCase}[1]{\rm (#1)}
\newcommand{\Block}[3]{{\mathcal #1}_{\hskip 1pt#2}^{#3}}
\newcommand{\LRegt}[1]{\mathcal{J}_{#1}}
\newcommand{\Rline}{\mathsf{R}}
\newcommand{\cX}[1]{\xi_{#1}}
\newcommand{\E}{\mathsf{E}}
\newcommand{\D}{\mathsf{D}}
\newcommand{\tsupLS}[1]{\widetilde{N}_{#1}}
\newcommand{\supLS}[1]{\widehat{N}_{#1}}
\newcommand{\Natural}{\mathsf{N}}
\newcommand{\PRline}{\mathsf{R}^{+}}
\newcommand{\cS}{c^{\ast}}
\newcommand{\T}[1]{T_{#1}}
\newcommand{\X}[1]{X_{#1}}
\newcommand{\Y}[1]{Y_{#1}}
\newcommand{\Sum}[2]{{#1}_{#2}}
\newcommand{\infLS}[1]{N_{#1}}
\newcommand{\McoR}[2]{{\matr{#1}}_{#2}}
\newcommand{\KRm}[3]{{Cor}\big(#1,#2,#3\big)}
\newcommand{\Det}{\mathsf{det}\,}
\newcommand{\matr}[1]{{{#1}}}
\newcommand{\kor}[2]{\rho_{#1}^{#2}}
\newcommand{\KondIbf}[1]{$\boldsymbol{#1}$}
\newcommand{\KondIit}[1]{${#1}$}
\newcommand{\KondIIbf}[2]{$\boldsymbol{#1}_{\boldsymbol{#2}}$}
\newcommand{\KondIIit}[2]{${#1}_{#2}$}
\newcommand{\KondIIIbf}[3]{$\boldsymbol{#1}_{\boldsymbol{(#2,#3)}}$}
\newcommand{\KondIIIit}[3]{${#1}_{(#2,#3)}$}

\begin{document}
\author[Vsevolod K. Malinovskii]{Vsevolod K. Malinovskii}

\keywords{Compound sums, Clustering{\zzz}type and queuing{\zzz}type models,
Distributional limit theorems, Refined approximations, Modular analysis,
Renewal theory, Risk theory, ergodic Markov and semi{\zzz}Markov systems.}

\address{Central Economics and Mathematics Institute (CEMI) of Russian Academy of Science,
117418, Nakhimovskiy prosp., 47, Moscow, Russia}

\email{admin@actlab.ru}

\urladdr{http://www.actlab.ru}

\title[]{\Large REFINED DISTRIBUTIONAL LIMIT THEOREMS\\[2pt]FOR COMPOUND SUMS}

\begin{abstract}{The paper is a sketch of systematic presentation of distributional
limit theorems and their refinements for compound sums. When analyzing, e.g.,
ergodic semi{\zzz}Markov systems with discrete or continuous time, this allows
us to separate those aspects that lie within the theory of random processes
from those that relate to the classical summation theory. All these limit
theorems are united by a common approach to their proof, based on the total
probability rule, auxiliary multidimensional limit theorems for sums of
independent random vectors, and (optionally) modular analysis.}
\end{abstract}

\maketitle

\section{Introduction}

In his review \cite{[Kesten=1977]} of Petrov's book \cite{[Petrov=1975]},
Kesten\footnote{Harry Kesten (1931{\zzz}2019), the Goldwin Smith Professor
Emeritus of Mathematics, ``whose insights advanced the modern understanding of
probability theory and its applications'' (quoted from obituary by Matt
Hayes).} wrote that ``in the period 1920{\zzz}1940 research in probability
theory was virtually synonymous with the study of sums of independent random
variables'', but ``already in the late thirties attention started shifting and
probabilists became more interested in Markov chains, continuous time
processes, and other situations in which the independence between summands no
longer applies, and today the center of gravity of probability theory has moved
away from sums of independent random variables.''

Nevertheless, Kesten continues, ``much work on sums of independent random
variables continues to be done, partly for their aesthetic appeal and partly
for the technical reason that many limit theorems, even for dependent summands,
can be reduced to the case of independent summands by means of various tricks.
The best known of these tricks is the use of regeneration points, or
`D{\oe}blin's trick'. In large part the fascination of the subject is due to
the fact that applications and the ingenuity of mathematicians continue to give
rise to new questions.'' Regarding applications, Kesten mentioned renewal
theory, the theory of optimal stopping, invariance principles and functional
limit theorems, fluctuation theory, and so on.

To summarize, Kesten argues that ``in all the years that these newer phenomena
were being discovered, people also continued working on more direct
generalizations and refinements of the classical limit theorems $\dots$ Many of
the proofs require tremendous skill in classical analysis $\dots$ and many of
the results, such as the Berry{\zzz}Esseen estimate, the Edgeworth expansion
and asymptotic results for large deviations, are important for statistics and
theoretical purposes.''

The paper is an overview that fits into the scheme outlined by Kesten. It is
focussed on refined ``distributional'' (or ``intrinsically
analytic''\footnote{The terms ``intrinsically analytic'' and
``measure{\zzz}theoretic'' applied to limit theorems are introduced in
\cite{[Chow=Teicher=1997]}. The former is equivalent (see, e.g.,
\cite{[Bingham=1989]}) to ``distributional'', or ``weak''.}, or ``weak'') limit
theorems for compound sums which are a natural extension of ordinary sums. They
are known as the key object of renewal theory and the core component of
D{\oe}blin's dissection of an ordinary sum defined on a Markov chain. By
``refined'' limit theorems we mean\footnote{The results for large deviations,
although interesting and available, are not presented here due to the space
limitation.} Berry{\zzz}Esseen's estimates and Edgeworth's expansions mentioned
by Kesten, in contrast to approximations without assessing their accuracy.

In contrast to ``measure{\zzz}theoretic'' limit theorems, for which
Kolmogorov's axiomatization is indispensable and which border on (or lie
within) the theory of random processes, ``distributional'' limit theorems,
including those for dependent random variables, are known long before the
axiomatization. For example, seeking in the 1900s, long before Kolmogorov's
advance, to demonstrate that independence is not a necessary condition in
``distributional'' weak law of large numbers (\cWLLN) and central limit theorem
(\cCLT), Markov introduced the Markov chains.

Without questioning the value of ``measure{\zzz}the\-ore\-tic'' limit theorems,
such as the strong law of large numbers (\cSLLN) and the law of iterated
logarithm (\cLIL), or, going further, the value of the theory of random
processes, we note that ``distributional'' rather than
``measure{\zzz}theoretic'' limit theorems are really needed for statistics,
cluster analysis, etc., where an infinite number of observations, clusters,
renewals, people in queue, and so on, never occur. Even in Bernoulli's
classical scheme, statistical inference is based (see detailed discussion in
\cite{[Chibisov=2016]}) on \WLLN and \CLT with its refinements, and not on
\SLLN and \cLIL.

Dealing with complex compound sums\footnote{See classification of compound sums
below. Modular analysis is not needed for simple compound sums.}, we focus on
modular analysis. Speaking of ``D{\oe}blin's trick'' which reduces ``many limit
theorems, even for dependent summands'' to the case of independent ones, Kesten
meant this approach. Bearing in mind the ladder technique, say ``Blackwell's
trick'' (see \cite{[Blackwell=1953]}), the modular analysis is even more widely
used. We will expand it to compound sums with general modular structure, but
note that (see, e.g., \cite{[Goetze=Hipp=1983]}, \cite{[Herve=Pene=2010]},
\cite{[Malinovskii=2021=b]}, \cite{[Rio=2017]}, \cite{[Tikhomirov=1981]}) this
is merely one, rather than the only and the best, method for studying such
sums.

Technically, ``D{\oe}blin's trick'' (or its counterpart ``Blackwell's trick'')
is not only the use of regeneration (or ladder) points. It implies (see, e.g.,
\cite{[Chung=1967]}) the use of Kolmogorov's inequality for maximum of partial
sums in order to move from D{\oe}blin's (or Blackwell's) dissection to ordinary
sum of independent modular summands. We call this approach ``basic technique''.

Delving deeper, the ``basic technique'' in its original form fails in local
theorems and yields no ``refined'' approximations, such as Berry{\zzz}Esseen's
estimates and Edgeworth's expansions. To be specific, dealing with a local
limit theorem for classical Markov chains, Kolmogorov (see
\cite{[Kolmogorov=1949]}) pointed that ``D{\oe}blin's method in its original
form'' is suitable ``only for proving integral theorems''. We quote\footnote{In
our translation from Russian. Apparently, this article has not been translated
in due course from Russian into English.} more from \cite{[Kolmogorov=1949]},
\S~2: ``$\dots$ the local theorems that form the main content of the present
paper will be obtained $\dots$ with the help of some strengthening of the
method that was developed by D{\oe}blin for proving the integral limit theorem
in the case of an infinite number of states. In order to make the development
of the method clear, we single out $\dots$ a brief exposition of D{\oe}blin's
method in its original form, in which it is suitable only for proving integral
theorems.''

In brief, Kolmogorov's advance in \cite{[Kolmogorov=1949]} was due to a
straightforward use of the total probability formula, or, according to his
terminology, ``basic identity''.

Further progress was made thirty years later in \cite{[Bolthausen=1980]} and in
a series of papers (see, e.g., \cite{[Bolthausen=1982=a]}, \cite{[Hipp=1985]},
\cite{[Malinovskii=1984]}--\cite{[Malinovskii=1991]}) that followed:
Berry{\zzz}Esseen's estimates and Edgeworth's expansions in \CLT for recurrent
Markov chain were obtained by using the total probability rule and the
auxiliary multidimensional limit theorems for sums of independent random
vectors. To distinguish this approach from ``basic technique'', or
``D{\oe}blin's method in its original form'', we call it ``advanced (modular,
if switching to modular summands is needed) technique''.

Being one of the main pillars of the advanced technique, the limit theorems for
sums of independent random vectors (see, e.g., \cite{[Bhattacharya=Rao=1976]},
\cite{[Dubinskaite=1982]}--\cite{[Dubinskaite=1984=b]}) are close to
perfection. In the advanced technique, analytical complexity, which never
disappears but flows from one form to another, is split into the use of,
firstly, the auxiliary theorems where (we quote from Preface of
\cite{[Bhattacharya=Rao=1976]}) ``precision and generality go hand in hand''
and, secondly, tedious but fairly standard classical analysis, e.g.,
approximation of integral sums by appropriate integrals and evaluation of
remainder terms.

Further presentation is arranged as follows. Section~\ref{zdfgerhgdf} is
devoted to genesis and classification of compound sums. In
Section~\ref{srfgfdgher}, we address elementary and refined renewal theorems
and refined limit theorems for cumulated rewards. Having formulated them as
limit theorems for simple compound sums, we outline their proof using advanced
technique instead of renewal equation, Laplace transforms, and Tauberian
theorems commonly used in renewal theory.

In Section~\ref{dstgfgjfgjrf}, we address complex compound sums ``with
irregular summation and independence'' related to the ruin problem of
collective risk theory. To get (refined) normal approximation when such sums
are proper, we use the modular analysis based on Blackwell's ladder idea. To
get (refined) quasi{\zzz}normal approximation when such sums are defective, we
resort to associated random variables. In this context, we touch upon the
inverse Gaussian approximation which deviates somewhat from the main topic of
this paper, but demonstrates that there are more sophisticated methods than
those related to asymptotic normality.

In Section~\ref{saergsdyherd}, seeking to use the full force of modular
analysis combined with advanced technique, we focus on compound sums with
general modular structure. Particular cases of this framework are complex
compound sums ``with irregular summation and dependence'', in particular Markov
dependence. The only case, although of significant practical interest, where
great impediments for asymptotic analysis arise is defective compound sums
``with irregular summation and dependence''. This case, which is tackled by
means of computer{\zzz}intensive analysis, is briefly discussed in
Section~\ref{srfydfuhjfg}. A concluding remark is made in
Section~\ref{rgryhdhf}.

The article, being a review, does not contain complete proofs. Most results are
formulated in a non{\zzz}strict manner, although the necessary references are
given everywhere. Striving for clarity, great emphasis is placed on discussing
the main ideas of the proofs. Overall, the aims of the advance described in
this review are the same as Lo{\`e}ve stated in \cite{[Loeve=1950]} writing on
limit theorems of probability theory: (i) to simplify proofs and forge general
tools out of the special ones, (ii) to sharpen and strengthen results, (iii) to
find general notions behind the results obtained and to extend their domains of
validity.

\section{Genesis and classification of compound sums}\label{zdfgerhgdf}

It was observed (see \cite{[Mikes=1970]}) that ``on the Continent, if people
are waiting at a bus{\zzz}stop, they loiter around in a seemingly vague
fashion. When the bus arrives, they make a dash for it $\dots$ An Englishman,
even if he is alone, forms an orderly queue of one.'' In terms of mathematical
modeling, this is a ``clustering{\zzz}type'' model opposed to a
``queuing{\zzz}type'' model. In the former, there is essentially no flow of
time, and the objects or events of interest are scattered in space. In the
latter, there is a flow of time, and the objects or events of interest form an
ordered queue.

In models of both types, the core is the compound sum
\begin{equation*}
\Sum{S}{\infLS{t}}=\sum_{i=1}^{\infLS{t}}\X{i}
\end{equation*}
with basis $(\T{i},\X{i})$, $i=1,2,\dots$, where $t>0$ and
\begin{equation}\label{xfghmhrey}
\infLS{t}=\inf\Big\{n\geqslant 1:\sum_{i=1}^{n}\T{i}>t\Big\}
\end{equation}
or $+\infty$, if the set is empty\footnote{Recall that $\inf\emptyset=+\infty$,
$\sup\emptyset= 0$.}. When $\infLS{t}=+\infty$, the sum $\Sum{S}{\infLS{t}}$ is
set equal to $+\infty$. For brevity, the real{\zzz}valued random variables
$\X{i}$ are called primary components ($p${\zzz}com\-po\-nents) and $\T{i}$
secondary components ($s${\zzz}com\-po\-nents) of the basis.

If the basis $(\T{i},\X{i})$, $i=1,2,\dots$, consists of independent random
vectors, then $\Sum{S}{\infLS{t}}$ is ``with independence''. Otherwise it is
``with dependence''. Examples of the latter are often found (see, e.g.,
\cite{[Pacheco=Prabhu=Tang=2009]}) in various models with Markov modulation. If
$s${\zzz}com\-po\-nents are positive, then $\Sum{S}{\infLS{t}}$ is ``with
regular summation''. Otherwise, it is ``with irregular summation''.

Apparently, compound sums ``with regular summation and independence'' fully and
distinctly came forth in ``queuing{\zzz}type'' renewal theory and its
ramifications. Compound sums ``with irregular summation and independence''
appeared (see, e.g., \cite{[Malinovskii=2021=c]}) in ruin problem and (see,
e.g., \cite{[Gut=2009]}) in many settings which involve random walks. On the
other hand, D{\oe}blin's dissection, i.e., switching from ordinary sum defined
on a Markov chain to modular simple compound sum ``with regular summation and
independence'', whose basis consists of intervals between regeneration points
and increments at these intervals, is another source of compound sums.

The complexity of ``irregular summation'' is rather transparent: the sums
$\Sum{V}{n}=\sum_{i=1}^{n}\T{i}$, $n=1,2,\dots$, are not progressively
increasing, whence $\{\infLS{t}=n\}$ is equal to $\{\Sum{V}{1}\leqslant
t,\dots,\Sum{V}{n-2}\leqslant t,\Sum{V}{n-1}\leqslant t<\Sum{V}{n}\}$, rather
than $\{\Sum{V}{n-1}\leqslant t<\Sum{V}{n}\}$, as would be for ``regular
summation''.

In ``dependence complexity'' case, the model has to be fleshed out. In the
search of non{\zzz}trivial examples which illustrate the general results, we
focus on Markov dependence and its extensions, but in fact any type of
dependence which allows an embedded modular structure with independent modules
suits us well. In terms of random processes, much of this topic is based on
successive ``starting over'', as time goes on, commonly called ``regeneration''
or (see \cite{[Kingman=1972]}, \cite{[Pacheco=Prabhu=Tang=2009]})
``regenerative phenomena''. With this approach, many cases that go beyond the
scope of Markov theory fall within the scope of this paper.

\begin{remark}\label{srgtfdhnfdhn}
Beyond the scope of this paper are the sums of the following two types. First,
the random sums $\Sum{S}{N}$ with $N$ and $\X{1},\X{2},\dots$ not necessarily
independent of each other, but with $N$ not of the form \eqref{xfghmhrey} or
the similar. Second, the random sums $\Sum{S}{N}=\sum_{i=1}^{N}\X{i}$ with
independence between $N$ and $\X{1},\X{2},\dots$. The former were introduced in
\cite{[Anscombe=1952]}, \cite{[Renyi=1957]}, \cite{[Renyi=1960]}, and
investigated, e.g., in \cite{[Gut=2009]}. The latter were introduced in
\cite{[Robbins=1948]} and thoroughly studied (see, e.g.,
\cite{[Gnedenko=Korolev=1996]}) in many works.\qed\end{remark}

\subsection{Proper and defective compound sums}\label{aergeryhr}

\begin{figure}[t]
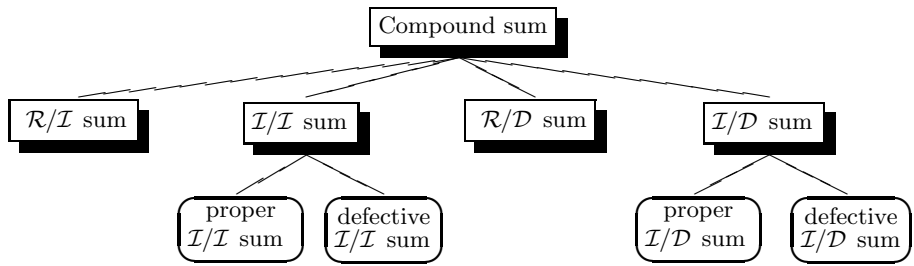

{\small\begin{center} \setlength{\GapWidth}{8pt}
\begin{bundle}{{
\newenvironment{Boxedminipage}
{\begin{Sbox}\begin{minipage}} {\end{minipage}\end{Sbox}\shadowbox{\TheSbox}}
\begin{Boxedminipage}{64pt}{\center{Compound sum}}\end{Boxedminipage}}}
\chunk{\shadowbox{
\newenvironment{Boxedminipage}
{\begin{Sbox}\begin{minipage}} {\end{minipage}\end{Sbox} {\TheSbox}}
\begin{Boxedminipage}{39pt}{\center{\mRI{\zzzx}sum}}\end{Boxedminipage}}}
\chunk{\begin{bundle}{\newenvironment{Boxedminipage}
{\begin{Sbox}\begin{minipage}} {\end{minipage}\end{Sbox}\shadowbox{\TheSbox}}
\begin{Boxedminipage}{37pt}
{\center{\mII{\zzzx}sum}}
\end{Boxedminipage}}
\chunk{\newenvironment{Boxedminipage}{\begin{Sbox}\begin{minipage}}
{\end{minipage}\end{Sbox} \Ovalbox{\TheSbox}}
\begin{Boxedminipage}{40pt}
{\center{proper\\[0pt]\mII{\zzzx}sum}}
\end{Boxedminipage}}
\chunk{\newenvironment{Boxedminipage}{\begin{Sbox}\begin{minipage}}
{\end{minipage}\end{Sbox}\Ovalbox{\TheSbox}}
\begin{Boxedminipage}{37pt}
{\center{defective\\[-2pt]\mII{\zzzx}sum}}
\end{Boxedminipage}}
\end{bundle}}
\chunk{\shadowbox{
\newenvironment{Boxedminipage}
{\begin{Sbox}\begin{minipage}} {\end{minipage}\end{Sbox} {\TheSbox}}
\begin{Boxedminipage}{40pt}{\center{\mRD{\zzzx}sum}}\end{Boxedminipage}}}
\chunk{\begin{bundle}{\newenvironment{Boxedminipage}
{\begin{Sbox}\begin{minipage}} {\end{minipage}\end{Sbox}\shadowbox{\TheSbox}}
\begin{Boxedminipage}{39pt}
{\center{\mID{\zzzx}sum}}\end{Boxedminipage}}
\chunk{\newenvironment{Boxedminipage}{\begin{Sbox}\begin{minipage}}
{\end{minipage}\end{Sbox} \Ovalbox{\TheSbox}}
\begin{Boxedminipage}{40pt}
{\center{proper\\[0pt]\mID{\zzzx}sum}}
\end{Boxedminipage}}
\chunk{\newenvironment{Boxedminipage}{\begin{Sbox}\begin{minipage}}
{\end{minipage}\end{Sbox} \Ovalbox{\TheSbox}}
\begin{Boxedminipage}{38pt}
{\center{defective\\[-2pt]\mID{\zzzx}sum}}
\end{Boxedminipage}}
\end{bundle}}
\end{bundle}
\end{center}}
\vskip 0pt\caption{Classification of compound sums.}\label{asrdtgyjhtty}
\end{figure}

The classification of compound sums, where shorthand for ``compound sum with
regular summation and independence'' is ``\mRI{\zzzx}sum'', for ``compound sum
with irregular summation and independence'' is ``\mII{\zzzx}sum'', for
``compound sum with regular summation and dependence'' is ``\mRD{\zzzx}sum'',
and for ``compound sum with irregular summation and dependence'' is
``\mID{\zzzx}sum'', is shown in Fig.~\ref{asrdtgyjhtty}. Under natural
regularity conditions, compound sums ``with regular summation'' are always
proper, while compound sums ``with irregular summation'' are either proper, or
defective\footnote{Recall (see \cite{[Feller=1971]}) that a random variable is
defective, with defect $1-p$, if it takes the value $\infty$ with probability
$p>0$. Otherwise, it is proper.}.

The following example shows that finding the defect of \mII{\zzzx}sum is not
easy.

\begin{example}\label{zesrfgrfh}
Let $(\X{i},\Y{i})$, $i=1,2,\dots$, be a subsidiary sequence of \iid random
vectors, whose components are positive and independent of each other. For $c>0$
and $\T{i}=\Y{i}-c\X{i}$, $\Sum{S}{\infLS{t}}=\sum_{i=1}^{\infLS{t}}\X{i}$ is
positive \mII{\zzzx}sum with basis $(\T{i},\X{i})$, $i=1,2,\dots$. While the
components of the subsidiary sequence are independent of each other, the
components of the basis are dependent on each other. If $\X{i}$ and $\Y{i}$,
$i=1,2,\dots$, are exponentially distributed with positive parameters $\paramX$
and $\paramY$, then (see, e.g., \cite{[Malinovskii=2021=b]}, Theorem~5.8)
\begin{equation}\label{sadfgbndfn}
\P\big\{\Sum{S}{\infLS{t}}\leqslant
x\big\}=\P\big\{\Sum{S}{\infLS{t}}<\infty\big\}-\frac{1}{\pi}\int_0^{\pi}f_{}(z,t,x)\,dz,
\end{equation}
where
\begin{equation}\label{34563745u856}
\P\big\{\Sum{S}{\infLS{t}}<\infty\big\}=\begin{cases}
1,&\paramX/(c\paramY)\geqslant 1,
\\[8pt]
\dfrac{\paramX}{c\paramY}\,\exp\big\{-t(c\paramY-\paramX)/c\big\},&\paramX/(c\paramY)<1
\end{cases}
\end{equation}
and
\begin{equation}\label{sargrsdyhe}
\begin{aligned}
f_{}(z,t,x)=&\,(\paramX/(c\paramY))(1+\paramX/(c\paramY)-2\sqrt{\paramX/(c\paramY)}\cos
z)^{-1}
\\[0pt]
&\times\exp\big\{t\paramY\,(\sqrt{\paramX/(c\paramY)}\cos z-1)
\\[0pt]
&-x\paramX
(c\paramY/\paramX)(1+\paramX/(c\paramY)-2\sqrt{\paramX/(c\paramY)}\cos z)\big\}
\\[0pt]
&\times(\cos(t\paramY\sqrt{\paramX/(c\paramY)}\sin z)
-\cos(t\paramY\sqrt{\paramX/(c\paramY)}\sin z+2z)).
\end{aligned}
\end{equation}
Consequently, $\Sum{S}{\infLS{t}}$ is proper if $\paramX/(c\paramY)\geqslant
1$, which is equivalent to $c\leqslant\cS=\E{Y}/\E{X}$, and defective
otherwise. Defect of $\Sum{S}{\infLS{t}}$ is yielded in \eqref{34563745u856}.
\qed\end{example}

Although the expression \eqref{34563745u856} for defect of $\Sum{S}{\infLS{t}}$
is complicated\footnote{This is a particular case of Cram{\'e}r's famous result
on the probability of ultimate ruin.}, the conditions under which
$\Sum{S}{\infLS{t}}$ is proper or defective, are simple. On the one hand, if
$\E\T{1}>0$, then (see Theorem~2 in \cite{[Feller=1971]}, Chapter~XII) the
ordinary sums $\Sum{V}{n}$ increase \as to $+\infty$, as $n\to\infty$.
Consequently, for any $t$ the summation limit $\infLS{t}$ is \as finite and
$\Sum{S}{\infLS{t}}$ is \as less than $+\infty$, whence proper. On the other
hand, if $\E\T{1}<0$, then the ordinary sums $\Sum{V}{n}$ decrease \as to
$-\infty$, as $n\to\infty$. With a positive probability, the maximum
$M=\max_{n\in\Natural}\Sum{V}{n}$ is finite and less than $t>0$ large enough,
whence $\infLS{t}$ and $\Sum{S}{\infLS{t}}$ are defective.

\subsection{Renewal theory and species of compound sums}\label{e56u6iuty}

In renewal theory, $p${\zzz}com\-po\-nents $\X{i}$, $i=1,2,\dots$, (typically
positive) are called ``rewards in the moments of renewals'' and
$s${\zzz}com\-po\-nents $\T{i}$, $i=1,2,\dots$, (always positive) are called
``intervals between successive renewals''. This interpretation yields a variety
of species of \mRI{\zzzx}sums. In particular, the total number of renewals
before time $t$, defined as $\tsupLS{t}=\sup\{n\geqslant 1:\Sum{V}{n}\leqslant
t\}$, or $0$, if $\T{1}>t$, comes to the fore. Focus shifts from $\infLS{t}$ to
$\tsupLS{t}$, which coincides with $\supLS{t}=\infLS{t}-1$ trivially related to
$\infLS{t}$.

When moving from relatively simple \mRI{\zzzx}sums to more complicated \mII\
and \mID{\zzzx}sums, both of them with ``irregular summation'', the difference
between altered (with summation limit $\tsupLS{t}$) and non{\zzz}altered (with
summation limit $\infLS{t}$) compound sums becomes non{\zzz}tri\-vial:
$\tsupLS{t}$ is no more trivially related to $\infLS{t}$ and the altered
compound sums $\Sum{S}{\tsupLS{t}}=\sum_{i=1}^{\tsupLS{t}}\X{i}$ and
$\Sum{S}{\supLS{t}}=\sum_{i=1}^{\supLS{t}}\X{i}$ are not bounded to coincide.

As for \mRD{\zzzx}sums with Markov dependence, which generalize \mRI{\zzzx}sums
in the same way as Markov renewal theory generalizes renewal theory, dependence
is introduced by embedded Markov chain $\{\cX{i}\}_{i\geqslant 0}$. In such
\mRD{\zzzx}sums, $s${\zzz}com\-po\-nents $\T{i}$, $i=1,2,\dots$, are the time
intervals between jumps of $\{\cX{i}\}_{i\geqslant 0}$ and $p${\zzz}components
$\X{i}$, $i=1,2,\dots$, which depend on $\{\cX{i}\}_{i\geqslant 0}$, are
rewards in the moments of renewals.

In this way, the variety of species of compound sums echoes the diversity of
renewal processes. There are\footnote{We refer to~\cite{[Cox=1970]},
Section~2.2, or \cite{[Cox=Miller=2001]}, Section~9.2. There are terminological
collisions, e.g., in \cite{[Karlin=Taylor=1975]}, Chapter~5, Section~7, this
process is called delayed renewal process.} modified, equilibrium, and ordinary
renewal processes. If $\T{1}$ is distributed differently than $\T{i}\eqD\T{}$,
$i=2,3,\dots$, then the renewal process is called modified. When all intervals
between renewals, including the first one, are identically distributed, the
renewal process is called ordinary. Furthermore, if
\begin{equation*}
\P\big\{\T{1}\leqslant
x\big\}=\frac{1}{\E\T{}}\int_{0}^{x}\big(1-\cdF{T}(z)\big)\,dz,\quad x>0,
\end{equation*}
where $\cdF{T}$ denotes \cdf of $\T{}$, then it is called equilibrium, or
stationary.

Regarding compound sums, modified are those in which a finite number of
elements of the basis $(\T{i},\X{i})$, $i=1,2,\dots$, e.g., the first element
$(\T{1},\X{1})$, differ from all the others. Unlike in renewal theory which
models technical repairable systems where the first repair does not necessarily
occur at the starting time zero, this modification looks awkward for
\mRI{\zzzx}sums: it boils down to a very special case of non{\zzz}identically
distributed summands. The similar modification of \mRD{\zzzx}sums with Markov
dependence comes down to the choice of the initial distribution of the embedded
Markov chain $\{\cX{i}\}_{i\geqslant 0}$ and is much more sensible.

It is noteworthy that, under natural regularity conditions, switching within
the species of compound sums, e.g., from altered to unaltered compound sums and
vice versa, does not affect the main{\zzz}term approximations, but affects
refinements such as Edgeworth's expansions.

\section{Renewal theory and limit theorems for \bmRI{\zzzx}sums}\label{srfgfdgher}

Founded (see \cite{[Feller=1940]}, \cite{[Feller=1941]}) in the 1940s as a
theoretical insight into technical repairable systems, the classical renewal
theory is focussed on ``queuing{\zzz}type'' models. The term ``renewal'', which
displaced the formerly used term ``industrial replacement'' coined (see
\cite{[Lotka=1939]}) in the 1930s by Lotka, is widely used nowadays, but was
not the main or default term even in Feller's seminal paper
\cite{[Feller=1949]} dated 1949.

In 1948, Doob \cite{[Doob=1948]} noted that ``renewal theory is ordinarily
reduced to the theory of certain types of integral equations $\dots$ However,
it is to be expected that a treatment in terms of the theory of probability,
which uses the modern developments of this theory, will shed new light on the
subject.'' Feller agreed (see~\cite{[Feller=1971]}, Chapter~VI) that in renewal
theory ``analytically, we are concerned merely with sums of independent
positive variables.''

\subsection{First appearance of the advanced technique}\label{erteruud}
In words, the renewal function\footnote{The asterisk denotes the convolution
operator.} $\RenU(t)=\sum_{n=0}^{\infty}\cdF{T}^{*n}(t)$, $t>0$, where
$\cdF{T}$ denotes \cdf of a positive random variable $\T{}$, is the expected
number of renewals in the time interval $(0,t)$, with the origin counted as
renewal epoch\footnote{In \cite{[Cox=Miller=2001]},
\cite{[Karlin=Taylor=1975]}, the renewal function is defined as
$\RenH(t)=\RenU(t)-1$, $t>0$ (or $\E\supLS{t}$), which differs from $\RenU(t)$
by one. In \cite{[Feller=1971]}, the function $\RenU(t)$, $t>0$, is called the
(ordinary) renewal process. In \cite{[Cox=Miller=2001]},
\cite{[Karlin=Taylor=1975]}, the renewal process is referred to as the
continuous{\zzz}time random process $\{\supLS{t}\}_{t\geqslant 0}$, i.e., the
random number of renewals in the time interval $(0,t)$, with the origin not
being counted as a renewal epoch.}. As a formula, this is
$\RenU(t)=1+\E\supLS{t}$. Together with altered \mRI{\zzzx}sum
$\Sum{S}{\supLS{t}}$, called in renewal theory ``cumulated rewards'', it is
built on \mRI{\zzzx}basis $(\T{i},\X{i})\eqD(\T{},\X{})$, $i=1,2,\dots$, with
\iid components and $\T{}$ positive.

The elementary renewal theorem (see, e.g., \cite{[Gut=2009]}), states that for
$0<\E\,\T{}\leqslant\infty$ the approximation\footnote{When $\E\,\T{}=\infty$,
the ratio $t/\E\,\T{}$ is replaced by $0$.}
\begin{equation}\label{xfghfdgds}
\E\supLS{t}=\frac{t}{\E\,\T{}}+\simb{o}(t),\quad t\to\infty,
\end{equation}
holds. The corresponding expansion up to vanishing term is
\begin{equation}\label{we567u65i}
\E\supLS{t}=\frac{t}{\E\,\T{}}+\frac{\D\T{}-(\E\,\T{}\,)^{2}}{2\,(\E\,\T{}\,)^{2}}
+\simb{o}(1),\quad t\to\infty,
\end{equation}
called refined elementary renewal theorem.

The elementary renewal theorem is widely known: see, e.g., Equality (3) in
\cite{[Cox=1970]}, Section~4.2, or Equality (17) in \cite{[Cox=Miller=2001]},
Section~9.2, or main formula in point (b) in \cite{[Karlin=Taylor=1975]},
Chapter~5, Section~6. In \cite{[Cox=Miller=2001]}, p.~345, it is noted that ``a
rigorous proof $\dots$ is possible under very weak assumptions about the
distribution of $T$, but requires difficult Tauberian arguments and will not be
attempted here''. It can also be obtained (with some caution with regard to the
definition of renewal function) from Theorem~1 in \cite{[Feller=1971]},
Chapter~XI, the proof of which is based on the renewal equation.

The refined elementary renewal theorem and the similar results for
higher{\zzz}order power moments of $\supLS{t}$ and $\Sum{S}{\supLS{t}}$,
including expansions up to vanishing terms, were studied (see, e.g.,
\cite{[Cox=1970]}, Section~4.5, and \cite{[Alsmeyer=1988]}) by various methods.

We outline a proof (see details in \cite{[Malinovskii=2021=b]}, Chapter~4) that
does not compete with the proofs mentioned above in the sense of elegance or
minimality of conditions. The advantage of this proof, based on the use of
refined \CLT and consisting of steps ~A--F, is that it will be routinely
carried over to much more complex settings.

\Step{A}{use of fundamental identity}
The identity
\begin{equation}\label{e56u7i67rt}
\begin{aligned}
\E\supLS{t}=&\,\sum_{n=1}^{\infty}n\;\P\big\{\supLS{t}=n\big\}
=\sum_{n=1}^{\infty}n\;\P\bigg\{\sum_{i=1}^{n}\T{i}\leqslant
t<\sum_{i=1}^{n+1}\T{i}\bigg\}
\\[-2pt]
=&\,\sum_{n=1}^{\infty}n\big(\cdF{T}^{*n}(t)-\cdF{T}^{*(n+1)}(t)\big)
\end{aligned}
\end{equation}
is obvious. If the probability density function (\cpdf) $\pdF{T}$ exists (we
assume this for simplicity of presentation), then \eqref{e56u7i67rt} can be
rewritten as
\begin{equation}\label{sadfgbnmcv}
\begin{aligned}
\E\supLS{t}=\sum_{n=1}^{\infty}n\int_{0}^{t}\pdF{T}^{*n}(t-z)
\,\P\big\{\T{n+1}>z\big\}\,dz.
\end{aligned}
\end{equation}

\Step{B}{reduction of range of summation}
At this step, we cut out from the sum on the right side of \eqref{e56u7i67rt}
those terms that correspond to small $n$ (i.e., $n<n_{t}$ with properly
selected $n_{t}\to\infty$, as $t\to\infty$). This is intuitively clear: for
large $t$, the random variable $\supLS{t}$ is likely to be large. Therefore,
the probabilities $\P\big\{\supLS{t}=n\big\}$ for small $n$ are small.
Technically, it is done using the well{\zzz}known probability inequalities,
such as Markov's inequality.

\Step{C}{reduction of range of integration}
At this step, relying on the moment conditions, we cut out from the integral in
\eqref{sadfgbnmcv} that part that corresponds to large $z$ (i.e., $z>z_{t}$
with properly selected $z_{t}\to\infty$, as $t\to\infty$). This is also
intuitively clear: any single random variable, including $\T{n+1}$, is small
compared to the sum $\sum_{i=1}^{n}\T{i}$, as $n$ is sufficiently large.

\Step{D}{application of refined \cCLT}
Writing $\parMu{T}{}=\E\,\T{}$, $\parSigma{T}{2}=\D\,\T{}$ and switching to
standardized random variables
$\stand{T}{i}=(\T{i}-\parMu{T}{})/\parSigma{T}{}$, $i=1,2,\dots$, we have
\begin{equation*}
\pdF{T}^{*n}(t-z)=\frac{1}{\parSigma{T}{}\sqrt{n}}\,\pdF{\,n^{-1/2}\sum_{i=1}^{n}
\stand{T}{i}}\big(\HNode{T}{n}{t-z}{}\big),
\end{equation*}
where $\HNode{T}{n}{t-z}{}=(t-z-\parMu{T}{}n)/(\parSigma{T}{}\sqrt{n})$.
Assuming that natural moment conditions are satisfied, we apply
Berry{\zzz}Esseen's estimate (see Theorem~11 in \cite{[Petrov=1975]},
Chapter~VII, \S~2) in the proof of \eqref{xfghfdgds} and Edgeworth's expansion
(see Theorem~17 in \cite{[Petrov=1975]}, Chapter~VII, \S~3) in the proof of
\eqref{we567u65i}. In both cases, these results are taken with non{\zzz}uniform
remainder terms.

\Step{E}{analysis of approximating term}
In the proof of \eqref{xfghfdgds}, where Berry{\zzz}Esseen's estimate was used,
the approximating term is
\begin{equation*}
\Approx{t}=\frac{1}{\parSigma{T}{}}\sum_{n>n_{t}}n^{1/2}
\int_{0}^{z_{t}}\Ugauss{0}{1}\big(\HNode{T}{n}{t-z}{}\big)
\,\P\big\{\T{}>z\big\}\,dz,
\end{equation*}
where $\Ugauss{0}{1}$ denotes \pdf of standard normal distribution. In the
proof of \eqref{we567u65i}, where Edgeworth's expansion was used, the
approximating term is the sum of
\begin{equation*}
\begin{aligned}
\SeqApprox{t}{\dag}=&\,\frac{1}{\parSigma{T}{}}\sum_{n>n_{t}}\sqrt{n}
\int_{0}^{z_{t}} \Ugauss{0}{1}(\HNode{T}{n}{t-z}{})\,\P\big\{\T{}>z\big\}\,dz
\\[0pt]
\SeqApprox{t}{\ddag}=&\,\frac{1}{\parSigma{T}{}}\frac{\E(\T{}^{3})}{6\,\parSigma{T}{3}}
\sum_{n>n_{t}}\int_{0}^{z_{t}}\,\big(\HNode{T}{n}{t-z}{3}
-3\,\HNode{T}{n}{t-z}{}\big)
\\[0pt]
&\times\Ugauss{0}{1}\big(\HNode{T}{n}{t-z}{}\big)\,\P\big\{\T{}>z\big\}\,dz.
\end{aligned}
\end{equation*}

Using Riemann summation formula with nodal points
$\HNode{T}{n}{t}{}{}<\HNode{T}{n-1}{t}{}{}<\dots<\HNode{T}{2}{t}{}{}<\HNode{T}{1}{t}{}{}$,
we seek to reduce $\Approx{t}$ to ${t}/{\E\,\T{}}$ with the required accuracy,
and $\SeqApprox{t}{\dag}+\SeqApprox{t}{\ddag}$ to
$t/\E\,\T{}+(\D\T{}-(\E\,\T{})^{2})/(2\,(\E\,\T{})^{2})$ with the required
accuracy. This is a tedious but fairly standard classical analysis left to the
reader.

\Step{F}{analysis of remainder term}
In the proof of \eqref{xfghfdgds}, the remainder term is
\begin{equation*}
\RemTerm{t}=\const{C}{1}{}\sum_{n>n_{t}}\delta_{n}\int_{0}^{z_{t}}
\big(1+|\,\HNode{T}{n}{t-z}{}\,|\,\big)^{-3}\,\P\big\{\T{}>z\big\}\,dz.
\end{equation*}
In the proof of \eqref{we567u65i}, the remainder term is
\begin{equation*}
\begin{aligned}
\RemTerm{t}=&\,\const{C}{2}{}\sum_{n>n_{t}}\frac{\delta_{n}}{\sqrt{n}}\int_{0}^{z_{t}}
\big(1+\big|\,\HNode{T}{n}{t-z}{}\,\big|\,\big)^{-4}\,\P\big\{\T{}>z\big\}\,dz,
\end{aligned}
\end{equation*}
with $\delta_{n}\to 0$, as $n\to\infty$. The former is $\simb{o}(t)$,
$t\to\infty$. The latter is $\simb{o}(1)$, $t\to\infty$, which is shown by
direct analytical methods. This is also a tedious but fairly standard classical
analysis left to the reader.

\subsection{Refined limit theorems for cumulated rewards}\label{srgrthjfg}

Moving from renewal function (or power moments of $\supLS{t}$) to the
distribution of altered \mRI{\zzzx}sum
$\Sum{S}{\supLS{t}}=\sum_{i=1}^{\supLS{t}}\X{i}$, called in renewal theory
``cumulated rewards'', we use the shorthand notation $\parMu{X}{}=\E\X{}$,
$\parSigma{X}{2}=\D\X{}$,
$\hkov{XT}{i,j}=\E((\X{}-\parMu{X}{})^{i}(\T{}-\parMu{T}{})^{j})$ for $i$ and
$j$ integers, and $\kov{XT}{}=\hkov{XT}{1,1}$. We put
\begin{equation*}
\MeanS{}{\!}=\parMu{X}{}\parMu{T}{-1},\quad
\StDevS{2}{}=\E\,(\parMu{X}{}T-\parMu{T}{}X)^{2}\parMu{T}{-3}.
\end{equation*}
By $\UGauss{0}{1}(x)$, we denote \cdf of standard normal distribution.

\begin{theorem}[Berry{\zzz}Esseen's estimate]\label{drgteghrweher}
If \pdf of $\T{}$ is bounded above by a finite constant\footnote{In
Theorem~\ref{drgteghrweher}, this condition is excessive. We repeat that we do
not strive for maximum generality (or rigor) in this presentation.},
$\StDevS{2}{}>0$, and $\E(\T{}^{3})<\infty$, $\E(\X{}^{3})<\infty$, then
\begin{equation*}
\sup_{x\in\Rline}\,\Big|\,\P\big\{\Sum{S}{\supLS{t}}-\MeanS{}{t}\leqslant
x\,\StDevS{}{t^{1/2}}\big\}-\UGauss{0}{1}(x)\,\Big|=\simb{O}\big(t^{-1/2}\big),\quad
t\to\infty.
\end{equation*}
\end{theorem}

\begin{theorem}[Edgeworth's expansion]\label{drsthtjmkfh}
If the conditions of Theorem~\ref{drgteghrweher} are satisfied and
$\E(\X{}^{k})<\infty$, $\E(\T{}^{k})<\infty$, then there exist polynomials
$\Polynom{Q}{r}(x)$, $r=1,2,\dots,k-3$, of degree $3r-1$, such that
\begin{equation*}
\begin{aligned}
&\sup_{x\in\Rline}\bigg|\,\P\big\{\Sum{S}{\supLS{t}}-\MeanS{}{t}\leqslant
x\StDevS{}{t^{1/2}}\big\}-\UGauss{0}{1}(x)
\\[-4pt]
&\hskip 90pt
-\sum_{r=1}^{k-3}\frac{\Polynom{Q}{r}(x)}{t^{r/2}}\,\,\Ugauss{0}{1}(x)\bigg|
=\simb{O}\big(t^{-(k-2)/2}\big),\quad t\to\infty.
\end{aligned}
\end{equation*}
In particular,
\begin{equation*}
\Polynom{Q}{1}(x)=-\frac{1}{6}\,\big(\Kompl{3}{0}\Herm{2}(x)+3\hskip
0.2pt\Value{I}{1}\big),
\end{equation*}
where $\Herm{2}(x)=x^{2}-1$ is Chebyshev{\zzz}Her\-mite's polynomial,
\begin{equation*}
\begin{aligned}
\Kompl{3}{0}=&\,\big(\big(\hkov{XT}{3,0}-3\parSigma{X}{2}\kov{XT}{}\parMu{T}{-1}+
6\parSigma{X}{2}\parSigma{T}{2}\parMu{T}{}\parMu{X}{-2}
-6\parSigma{T}{2}\kov{XT}{}\parMu{X}{2}\parMu{T}{-3}\big)\,
\parMu{T}{-1}
\\[2pt]
&-3\,\big(\hkov{XT}{2,1}-2\kov{XT}{2}\parMu{T}{-1}
+\parSigma{X}{2}\parSigma{T}{2}\parMu{T}{-1}\big)\,
\parMu{T}{}\parMu{X}{-2}
\\[2pt]
&+3\,\big(\hkov{XT}{1,2}-\parSigma{T}{2}\kov{XT}{}\parMu{T}{-1}\big)\,\parMu{X}{2}\parMu{T}{-3}
\\[2pt]
&-\big(\hkov{XT}{0,3}-3\parSigma{T}{4}\parMu{T}{-1}\big)\,\parMu{X}{3}\parMu{T}{-4}\big)\,\StDevS{-3}{},
\end{aligned}
\end{equation*}
and
$\Value{I}{1}=\parMu{X}{}\big(\parSigma{T}{2}\,\parMu{T}{-2}+1\big)\,\StDevS{-1}{}$.
\end{theorem}

Both Theorems~\ref{drgteghrweher} and \ref{drsthtjmkfh} are proved (see details
in \cite{[Malinovskii=2021=b]}, Chapter~4) by means of practically the same
advanced technique which was used in the proof of \eqref{xfghfdgds} and
\eqref{we567u65i}, respectively. It consists of steps~A--F, which we will draw.

\Step{A}{use of fundamental identity}
Equalities \eqref{e56u7i67rt} and \eqref{sadfgbnmcv} are replaced by
\begin{equation}\label{srdtuytfgik}
\begin{aligned}
\P\big\{\Sum{S}{\supLS{t}}\leqslant
x\big\}=&\,\sum_{n=1}^{\infty}\int_{0}^{t}\P\bigg\{\sum_{i=1}^{n}\X{i}\leqslant
x\,\bigg|\,\sum_{i=1}^{n}\T{i}=t-z\,\bigg\}
\\[2pt]
&\times\pdF{T}^{*n}(t-z)\,\P\big\{\T{n+1}>z\big\}\,dz,
\end{aligned}
\end{equation}
which is the total probability formula.

\Steps{B and C}
remain essentially the same.

\Step{D}{application of refined \cCLT}
The integrand in \eqref{srdtuytfgik} is ready for application of (see
\cite{[Dubinskaite=1982]}--\cite{[Dubinskaite=1984=b]}) two{\zzz}dimensional
hybrid integro{\zzz}local (for density) \cCLT. In more detail, in the proof of
Theorem~\ref{drgteghrweher}, we use Berry{\zzz}Esseen estimate, in the proof of
Theorem~\ref{drsthtjmkfh}, we use Edgeworth's expansion in this \cCLT, both
with non{\zzz}uniform remainder terms.

\Steps{E and F}
remain nearly the same. Technically, they are more complicated than steps~E
and~F in the proof of \eqref{xfghfdgds} and \eqref{we567u65i}, but use
essentially the same analytical techniques.

\subsection{Garbage and deficiency of basic technique}\label{dtyiugoh}

When studying altered \mRI{\zzzx}sum $\Sum{S}{\supLS{t}}$, the main point of
basic technique consists in approximating $\Sum{S}{\supLS{t}}$ by the ordinary
sum $\Sum{S}{\,[\E\!\supLS{t}]}$, whose asymptotical normality is evident: the
real number $[\E\!\supLS{t}]$ is (see \eqref{xfghfdgds}) of the order
$t/\E\,\T{}\to\infty$, as $t\to\infty$, and the summands are \iid In other
words, although technically the basic technique applies Kolmogorov's inequality
for maximum of partial sums, its fundamental idea is to throw the difference
$\Garb{t}=\Sum{S}{\supLS{t}}-\Sum{S}{\,[\E\!\supLS{t}]}$, called garbage, in
the trash.

Constrained by this idea, the basic technique does not allow further
refinements, such as in Theorems~\ref{drgteghrweher} and \ref{drsthtjmkfh},
since $\Garb{t}$ is a random variable of order $t^{1/4}$: under natural
regularity conditions, including $\E(\T{}^{3})<\infty$, $\E(\X{}^{3})<\infty$,
\begin{equation}\label{asrghher}
\begin{aligned}
&\sup_{x\in\Rline}\bigg|\,
\P\bigg\{\frac{\parMu{T}{3/4}}{\parSigma{X}{}\parSigma{T}{1/2}t^{1/4}}
\,\Garb{t}\leqslant x\bigg\}
\\[0pt]
&\hskip 40pt
-\int_{-\infty}^{\infty}\UGauss{0}{1}\bigg(\frac{x}{\sqrt{|z|}}\bigg)
\Ugauss{0}{1}(z)\,dz\bigg|=\simb{O}\big(t^{-1/4}\big),\quad t\to\infty.
\end{aligned}
\end{equation}

The proof of \eqref{asrghher} (see details in \cite{[Malinovskii=1987=b]} and
\cite{[Malinovskii=1988]}) by means of advanced technique follows the scheme
sketched in the proof of \eqref{xfghfdgds} and \eqref{we567u65i} and
Theorems~\ref{drgteghrweher} and \ref{drsthtjmkfh}. It consists of steps~A--F,
applies refined \CLT with non{\zzz}uniform remainder terms, and requires mere
standard classical analysis at steps~B, C, E, and F.

\begin{figure}[t]
\center{\includegraphics[scale=.8]{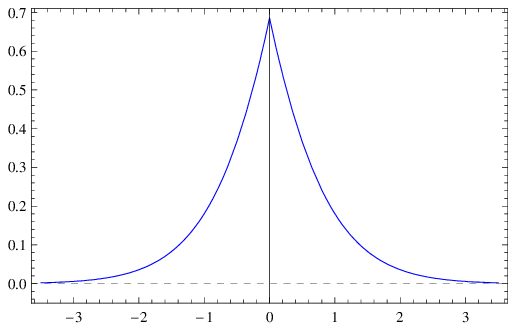}}
\caption{Graph of $f(y)={\displaystyle\int_{-\infty}^{\infty}}
\Ugauss{0}{1}\Big(\frac{y}{\sqrt{|z|}}\Big)\frac{1}{\sqrt{|z|}}\,\Ugauss{0}{1}(z)\,dz$.}\label{rthrurehd}
\end{figure}

For greater clarity, note that \pdf of the limit distribution in
\eqref{asrghher} is
\begin{equation*}
f(y)=\int_{-\infty}^{\infty}\Ugauss{0}{1}\bigg(\frac{y}{\sqrt{|z|}}\bigg)
\frac{1}{\sqrt{|z|}}\;\Ugauss{0}{1}(z)\,dz,
\end{equation*}
its maximal value is
\begin{equation*}
f(0)=\frac{1}{\sqrt{2\pi}}\int_{-\infty}^{\infty}\frac{1}{\sqrt{|z|}}\;\Ugauss{0}{1}(z)\,dz
=\frac{\Gamma(1/4)}{2^{3/4}\pi}\approx 0.686,
\end{equation*}
and its graph is drawn in Fig.~\ref{rthrurehd}.

\section{Limit theorems for \bmII{\zzzx}sums and collective risk theory}\label{dstgfgjfgjrf}

The reader familiar with collective risk theory readily recognizes in
$\P\{\Sum{S}{\infLS{t}}\leqslant x\}$ of Example~\ref{zesrfgrfh} the
probability of ruin within time $x$, where\footnote{Under the assumption of
independence from each other.} $\X{i}$, $i=1,2,\dots$, are intervals between
claims, $\Y{i}$, $i=1,2,\dots$, are claim amounts, $t$ is the initial capital,
and $c$ is the premium intensity. For $0<c\leqslant\cS$, the \mII{\zzzx}sum
$\Sum{S}{\infLS{t}}$ is proper, while (in terms of the risk reserve model)
insurance is ruinous. For $c>\cS=\E{Y}/\E{T}$, the \mII{\zzzx}sum
$\Sum{S}{\infLS{t}}$ is defective, while insurance is profitable.

\subsection{Modular approach: a way to manage complexity}

A system is modular if its components can be separated and recombined, with the
advantage of flexibility and variety in use. Breaking such a system into
independent or nearly independent modules is done in order (see
\cite{[Baldwin=Clark=2000]}) ``to hide the complexity of each part behind an
abstraction and interface''. Modularity is implemented in two ways: by
splitting an orderly queue (for ``queuing{\zzz}type'' models) by events that
occur one after another in time, or by sorting spatial items (for
``clustering{\zzz}type'' models) into loosely related and similar sets or
groups.

An example of the latter is the development of the OS/360 operating system for
the original IBM 360 line of computers. At first (see \cite{[Brooks=1975]}), it
was organized in a relatively indecomposable way. It was deemed that each
programmer should be familiar with all the material, i.e., should have a copy
of the workbook in his office. But over time, this led to a rapid growth of
troubles. For example, even maintaining the workbook, whose volume constantly
and significantly increased, began to take up significant time from each
working day. Therefore, it became necessary to figure out whether programmers
developing one particular module really need to know about the structure of the
entire system.

It was concluded that, especially in large projects, the answer is negative and
managers (quoting from \cite{[Parnas=1972]}) ``should pay attention to
minimizing interdependencies. If knowledge is hidden or encapsulated within a
module, that knowledge cannot affect, and therefore need not be communicated
to, other parts of a system.'' This idea has been framed as ``information
hiding'', a key concept in the modern object{\zzz}oriented approach to computer
programming.

While the modular approach is not needed for \mRD{\zzzx}sums, it is useful when
studying \mII, \mRD, and proper \mID{\zzzx}sums. In ``queuing{\zzz}type''
models, it is applied by splitting an orderly queue by events (see, e.g.,
\cite{[Pacheco=Prabhu=Tang=2009]}: ``regenerative phenomena'', or the like)
that occur one after another in time. In ``clustering{\zzz}type'' models, it is
applied by sorting spatial items into loosely related (and similar) sets or
groups.

\subsection{Normal approximation for proper \bmII{\zzzx}sums}\label{asertewtyrh}

In this case, modular structure is defined by the random indices
$0\equiv\Ladd{0}<\Ladd{1}<\Ladd{2}\dots$, where
\begin{equation*}
\Ladd{k}=\inf\big\{i>\Ladd{k-1}:\Sum{V}{i}>\Sum{V}{\Ladd{k-1}}\big\}\eqD\Ladd{},\quad
k=1,2,\dots.
\end{equation*}
The key point is the following ``exact''\footnote{This means that, unlike,
e.g., D{\oe}blin's dissection, it does not contain incomplete initial and final
modules.} partition of the proper \mII{\zzzx}sum $\Sum{S}{\infLS{t}}$, called
Blackwell's dissection\footnote{This dissection was introduced by Blackwell in
\cite{[Blackwell=1953]}.}:
\begin{equation}\label{ew5t6rtjuty}
\Sum{S}{\infLS{t}}=\sum_{i=1}^{\BinfLS{t}}\Block{X}{i}{},
\end{equation}
where $\BinfLS{t}=\inf\big\{n>0:\sum_{i=1}^{n}\Block{T}{i}{}>t\big\}$, or
$\infty$, if the set is empty,
$\Block{X}{i}{}=\sum_{j=\Ladd{i-1}+1}^{\Ladd{i}}\X{j}$, and
$\Block{T}{i}{}=\sum_{j=\Ladd{i-1}+1}^{\Ladd{i}}\T{j}>0$ by definition. Bearing
in mind that $t$ can be exceeded only at a ladder index, equality
\eqref{ew5t6rtjuty} is obvious.

Since $(\Block{T}{i}{},\Block{X}{i}{})\eqD(\Block{T}{}{},\Block{X}{}{})$,
$i=1,2,\dots$, are \iid and $\Block{T}{}{}$ is positive by definition, the sum
on the right{\zzz}hand side of \eqref{ew5t6rtjuty} is a modular \mRI{\zzzx}sum
which is examined as above. Using \eqref{ew5t6rtjuty}, nearly all results on
asymptotic normality and its refinements available for \mRI{\zzzx}sums can be
transferred to proper \mII{\zzzx}sums.

In particular (see details in \cite{[Malinovskii=2021=b]}, Chapter~9), in
Example~\ref{zesrfgrfh} with $0<c<\cS$ the normal approximation under natural
regularity conditions is
\begin{equation}\label{asertgrfyhd}
\lim_{t\to\infty}\sup_{x\in\Rline}\,\big|\,\P\big\{\Sum{S}{\infLS{t}}\leqslant
x\big\}-\UGauss{\mominus t}{\Dominus^{2} t}(x)\,\big|=0,
\end{equation}
where $\mominus=\parMu{X}{}\parMu{T}{-1}$ and
$\Dominus^{2}=\E(\parMu{\T{}}{}\X{}-\parMu{\X{}}{}\T{}\,)^{2}\parMu{\T{}}{-3}$.
Its refinements, i.e., Berry{\zzz}Esseen's estimate and Edgeworth's expansion,
are similar to those in Theorems~\ref{drgteghrweher} and \ref{drsthtjmkfh} and
are omitted in this presentation.

Regarding $\mominus$ and $\Dominus^{2}$ expressed above in the original (rather
than modular) terms, the following should be noted. The refined normal
approximation \eqref{asertgrfyhd} obtained by modular approach is first written
in terms of the modular random variables $\Ladd{}$, $\Block{T}{}{}$, and
$\Block{X}{}{}$. Converting them to the original terms requires additional
effort. There are several ways to do this. First, Wald's identities can bring a
great release since certain combinations of moments of block random variables
$\Block{T}{}{}$ and $\Block{X}{}{}$ can be represented through the moments of
the original random variables $\T{}$ and $\X{}$. Second, moments of ladder
index $\Ladd{}$ and ladder modules $\Block{T}{}{}$, $\Block{X}{}{}$ can be
calculated analytically, using (see, e.g., \cite{[Feller=1971]}) Spitzer's
sums.

\subsection{Quasi{\zzz}normal approximation for defective
\bmII{\zzzx}sums}\label{srdtdttfdhere}

For $c>\cS$, under natural regularity conditions (see details in
\cite{[Malinovskii=2021=b]}, Chapter~9), quasi{\zzz}normal
approximation\footnote{This name emphasizes the presence of a normal
distribution function. In risk theory, it is known as
Cram{\'e}r{\zzz}Lundberg's approximation.} is
\begin{equation}\label{dfvdsfbdfbXX}
\lim_{t\to\infty}\sup_{x>0}\,\big|\,e^{\LunAdjust
t}\,\P\big\{\Sum{S}{\infLS{t}}\leqslant x\big\}-\LunKonst\;\UGauss{\moplus
t}{\Doplus^{2}t}(x)\,\big|=0.
\end{equation}
It follows from \eqref{dfvdsfbdfbXX} that
$\P\{\Sum{S}{\infLS{t}}<\infty\}\approx\,\LunKonst\,e^{-\LunAdjust t}$, which
is an asymptotic formula for defect (cf. \eqref{34563745u856}) of defective
\mII{\zzzx}sum $\Sum{S}{\infLS{t}}$.

Here $\LunAdjust$ is a positive solution (\wrt $x$) to the nonlinear
equation\footnote{In risk theory, equation \eqref{sdfgrhjrgdeXX} is called
Lundberg's equation. Its positive solution $\LunAdjust$ is called Lundberg's
exponent, or adjustment coefficient.}
\begin{equation}\label{sdfgrhjrgdeXX}
\E{\,e^{xT}\,}=1
\end{equation}
and in terms of the associated random variables $\AssT{i}\eqD\AssT{}$ and
$\AssX{i}\eqD\AssX{}$, $i=1,2,\dots$,
\begin{equation*}
\begin{aligned}
&\moplus=\parMu{\AssX{}}{}\parMu{\AssT{}}{-1},
\quad\Doplus^{2}=\E\big(\parMu{\AssT{}}{}\,\AssX{}-\parMu{\AssX{}}{}\,\AssT{}\big)^{2}\parMu{\AssT{}}{-3},
\\[4pt]
&\LunKonst=\frac{1}{\LunAdjust\,\parMu{\AssT{}}{}}\exp\bigg\{-\sum_{n=1}^{\infty}\frac{1}{n}\,
\P\big\{\Sum{V}{n}>0\big\}
-\sum_{n=1}^{\infty}\frac{1}{n}\,\P\big\{\AssSum{V}{n}\leqslant 0\big\}\bigg\},
\end{aligned}
\end{equation*}
where $\AssSum{V}{n}=\sum_{i=1}^{n}\AssT{i}$. Recall (see, e.g., Example~(b) in
\cite{[Feller=1971]}, Chapter~XII, Section~4) that the associated random
variables we define as follows. Starting with \mII{\zzzx}basis
$(\T{i},\X{i})\eqD(\T{},\X{})$, $i=1,2,\dots$, we switch from
$\cdF{TX}(t,x)=\P\big\{T\leqslant t,X\leqslant x\big\}$ to
$\cdF{\AssT{}\AssX{}}(t,x)=\P\big\{\AssT{}\leqslant t,\AssX{}\leqslant x\big\}$
defined by the integral $\int_{-cx}^{t}\int_{0}^{x}e^{\LunAdjust
u}\,\cdF{TX}(du,dv)$. The latter probability distribution is proper and
$\E\,\AssT{}>0$. Commonly used shorthand for it is
$\cdF{\AssT{}\AssX{}}(dt,dx)=e^{\LunAdjust z}\cdF{TX}(dt,dx)$.

\begin{remark}
The proof of \eqref{dfvdsfbdfbXX} carried out using the modular approach is
based (see \cite{[Bahr=1974]} and \cite{[Malinovskii=1994=c]}) on the following
counterpart of equality \eqref{ew5t6rtjuty}:
\begin{equation}\label{sdfvbfdhnfZZ}
\P\big\{\Sum{S}{\infLS{t}}\leqslant x\}
=\E\bigg(\exp\bigg\{-\LunAdjust\;\sum_{i=1}^{\infLSAss{t}}\,\AssBlock{T}{i}{}\bigg\}\,
\1_{(-\infty,x]}\bigg(\sum_{i=1}^{\infLSAss{t}}\,\AssBlock{X}{i}{}\bigg)\bigg),
\end{equation}
where $\infLSAss{t}=\inf\big\{n\geqslant
1:\sum_{i=1}^{n}{\AssBlock{T}{i}{}}>t\big\}$, or $\infty$, if the set is empty,
and the simple modular basis $(\AssBlock{T}{k}{},\AssBlock{X}{k}{})$,
$k=1,2,\dots$, is given by
\begin{equation*}
\AssBlock{T}{k}{}=\Sum{\Ass{V}{}}{\LaddAss{k}}-\Sum{\Ass{V}{}}{\LaddAss{k-1}}>0\quad\text{and}\quad
\AssBlock{X}{k}{}=\Sum{\Ass{S}{}}{\LaddAss{k}}-\Sum{\Ass{S}{}}{\LaddAss{k-1}},\quad
k=1,2,\dots,
\end{equation*}
where $\AssSum{S}{n}=\sum_{i=1}^{n}\AssX{i}$ and the ladder indices generated
by the associated random variables are $\LaddAss{0}\equiv 0$ and
$\LaddAss{k}=\inf\big\{i>\LaddAss{k-1}:\Sum{\Ass{V}{}}{i}>\Sum{\Ass{V}{}}{\LaddAss{k-1}}\big\}$,
$k=1,2,\dots$.\qed\end{remark}

\begin{remark}
Edgeworth's expansions are known for both the normal approximation
\eqref{asertgrfyhd} and (see \cite{[Malinovskii=1994=c]}) the quasi{\zzz}normal
approximation \eqref{dfvdsfbdfbXX}. Further details see, e.g., in
\cite{[Malinovskii=2021=b]}, Chapter~9.\qed\end{remark}

\begin{example}[\textsl{Example~\ref{zesrfgrfh} continued}]\label{srtdsyhtfjgt}
Let the random variables $\X{i}\eqD\X{}$ and $\Y{i}\eqD\Y{}$, $i=1,2,\dots$, be
exponentially distributed with parameters $\paramX>0$ and $\paramY>0$. Clearly,
$\E{\,e^{-x cX}}=\paramX/(\paramX+cx)$ and $\E{\,e^{xY}}=\paramY/(\paramY-x)$.
It follows that $\E{\,e^{xT}\,}=\E{\,e^{xY}}\E{\,e^{-x cX}}$, whence
$\E{\,e^{xT}\,}=\paramY\paramX/((\paramY-x)(\paramX+cx))$, and equation
\eqref{sdfgrhjrgdeXX} is the quadratic equation
$\paramY\paramX/((\paramY-x)(\paramX+cx))=1$ with respect to $x$. For
$c>\cS=\paramX/\paramY$, its positive solution is
$\LunAdjust=\paramY\,(1-\paramX/(c\paramY))$. Further calculations yield
$\LunKonst=\paramX/(c\paramY)$ and
\begin{equation*}
\begin{gathered}
\mominus=-\frac{1}{c\,(1-\paramX/(c\paramY))},\quad
\Dominus^{2}=-\frac{2\,(\paramX/(c\paramY))}{c^{2}\paramY\,(1-\paramX/(c\paramY))^{3}},
\\[0pt]
\moplus=\frac{\paramX/(c\paramY)}{c\,(1-\paramX/(c\paramY))},\quad
\Doplus^{2}=\frac{2\,(\paramX/(c\paramY))}{c^{2}\paramY\,(1-\paramX/(c\paramY))^{3}}.
\end{gathered}
\end{equation*}

\begin{figure}[t]
\center{\includegraphics[scale=.8]{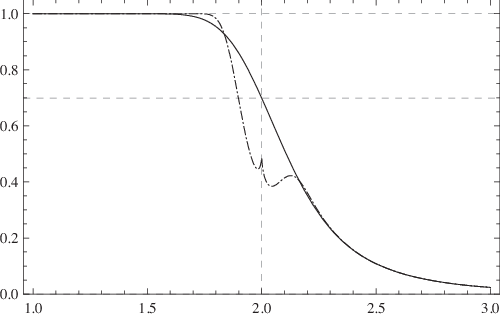}}
\caption{Graph ($X$-axis is $c$) of $\P\big\{\Sum{S}{\infLS{t}}\leqslant
x\big\}$ (solid line), which is a monotone decreasing function calculated
numerically, using equalities \eqref{sadfgbndfn}--\eqref{sargrsdyhe}, and the
approximation \eqref{asertgrfyhd}, \eqref{dfvdsfbdfbXX} (dash{\zzz}dotted
line), when $T$ and $Y$ are exponentially distributed with parameters
$\paramT=2$ and $\paramY=1$, respectively, and $t=200$, $u=10$. Horizontal
line: $\P\big\{\Sum{S}{\infLS{t}}\leqslant x\big\}$ with $c$ equal to $\cS$,
which is $0.699$.}\label{asrdtyjtyt6}
\end{figure}
\begin{figure}[t]
\center{\includegraphics[scale=.8]{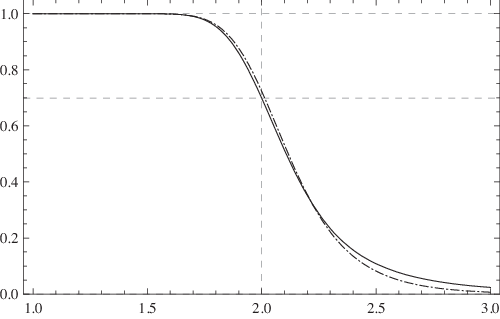}}
\caption{Graphs ($X$-axis is $c$) of $\P\big\{\Sum{S}{\infLS{t}}\leqslant
x\big\}$ (solid line) calculated numerically, using equalities
\eqref{sadfgbndfn}--\eqref{sargrsdyhe}, and approximation \eqref{asergfhrXX}
(dash{\zzz}dotted line), when $T$ and $Y$ are exponentially distributed with
parameters $\paramT=2$ and $\paramY=1$, respectively, and $t=200$, $u=10$.
Horizontal line: $\P\big\{\Sum{S}{\infLS{t}}\leqslant x\big\}$ with $c$ equal
to $\cS$, which is $0.699$.}\label{sdtuyikghlhj}
\end{figure}

Approximations \eqref{asertgrfyhd} and \eqref{dfvdsfbdfbXX} are compared (see
Fig.~\ref{asrdtyjtyt6}) with the exact values of
$\P\big\{\Sum{S}{\infLS{t}}\leqslant x\big\}$ calculated numerically. They are
unsatisfactory in the vicinity of the point $\cS$ and satisfactory elsewhere.
\qed\end{example}

\begin{remark}
The assumption that a positive solution $\LunAdjust$ to the equation
\eqref{sdfgrhjrgdeXX} exists is a serious constraint on $\T{}=\Y{}-c\X{}$. It
follows that $\E{\,e^{xT}\,}$ must exist in the right neighborhood of $0$ and
that $\P\big\{T>x\big\}$ is exponentially bounded from above for $x>0$. The
latter follows from Markov's inequality, as follows:
$\P\big\{T>x\big\}\leqslant e^{-\LunAdjust x}\,\E{\,e^{\LunAdjust
T}}=e^{-\LunAdjust x}$.\qed\end{remark}

\subsection{Inverse Gaussian approximation for \bmII{\zzzx}sums}\label{aserterutjt}

The asymptotic (as $t\to\infty$) behavior of \mII{\zzzx}sum does not
necessarily have to be related to a normal distribution. The following inverse
Gaussian approximation (see details in \cite{[Malinovskii=2021=b]},
\cite{[Malinovskii=2021=c]}) for $\P\{\Sum{S}{\infLS{t}}\leqslant x\}$ has
great advantages. In the framework of Example~\ref{zesrfgrfh}, introduce
$M=\E{\X{}}/\E{\Y{}}$ and $D^{\,2}=
((\E{\X{}})^{2}\D{\Y{}}+(\E{\Y{}})^{2}\D{\X{}})/(\E{\Y{}})^{3}$. Under certain
regularity conditions, we have
\begin{equation}\label{asergfhrXX}
\lim_{t\to\infty}\sup_{x>0}\big|\,\P\big\{\Sum{S}{\infLS{t}}\leqslant
x\big\}-\AInt{M}{t,c}(x)\,\big|=0,
\end{equation}
where
\begin{equation*}
\begin{aligned}
\AInt{M}{t,c}(x)&=\int_{0}^{\frac{cx}{t}}\frac{1}{z+1}\,
\Ugauss{cM(z+1)}{\frac{c^{2}D^{2}}{u}(z+1)}(z)\,dz,
\end{aligned}
\end{equation*}
or, in equivalent form,
\begin{equation*}
\begin{aligned}
\AInt{M}{t,c}(x)&=\begin{cases} \big(F(z+1;\muIG,\lambdaIG)
\\[2pt]
\hskip 30pt -F(1;\muIG,\lambdaIG)
\big)\Big|_{z=\frac{cx}{t},\muIG=\frac{1}{1-cM},\lambdaIG=\frac{t}{c^{2}D^{2}}},&
0<c\leqslant\cS,
\\[2pt]
\exp\bigg\{-\dfrac{2\lambdaIG}{\HmuIG}\bigg\}\,\big(F(z+1;\HmuIG,\lambdaIG)
\\[6pt]
\hskip 30pt
-F(1;\HmuIG,\lambdaIG)\big)\Big|_{z=\frac{cx}{t},\HmuIG=\frac{1}{cM-1},
\lambdaIG=\frac{t}{c^{2}D^{2}}},&c>\cS
\end{cases}
\end{aligned}
\end{equation*}
and
\begin{equation}\label{sadfgbfnfXX}
\begin{aligned}
F(z;\muIG,\lambdaIG)=&\,
\UGauss{0}{1}\bigg(\sqrt{\frac{\lambdaIG}{z}}\,\bigg(\frac{z}{\muIG}-1\bigg)\bigg)
\\[0pt]
&+\exp\bigg\{\frac{2\lambdaIG}{\muIG}\bigg\}
\,\UGauss{0}{1}\bigg(-\sqrt{\frac{\lambdaIG}{z}}\,\bigg(\frac{z}{\muIG}+1\bigg)\bigg),\quad
z>0.
\end{aligned}
\end{equation}

The approximation \eqref{asergfhrXX} is called inverse Gaussian because
\eqref{sadfgbfnfXX} is \cdf of the inverse Gaussian distribution with
parameters $\muIG>0$ and $\lambdaIG>0$.

\begin{remark}
In \cite{[Malinovskii=2021=b]}, Chapter~7, refinements (i.e., analogues of the
Berry{\zzz}Esseen estimate and the Edgeworth expansion) of the inverse Gaussian
approximation \eqref{asergfhrXX} are obtained.\qed\end{remark}

Comparing approximations \eqref{asertgrfyhd} and \eqref{dfvdsfbdfbXX} with
\eqref{asergfhrXX}, one can argue\footnote{Using heuristic arguments and
referring to exact formulas \eqref{sadfgbndfn}--\eqref{sargrsdyhe}.} that the
results for proper and defective \mII{\zzzx}sums in Example~\ref{zesrfgrfh}
must border seamlessly. But the normal (with $0<c<\cS=\E{Y}/\E{T}$) and
quasi{\zzz}normal (with $c>\cS$) approximations \eqref{asertgrfyhd} and
\eqref{dfvdsfbdfbXX} do not border seamlessly. Moreover (see
Fig.~\ref{asrdtyjtyt6}), both of them are either poor, or
invalid\footnote{Quasi{\zzz}normal approximation \eqref{dfvdsfbdfbXX} formally
requires $c>\cS$. The existence of $\LunAdjust>0$ is possible only if the tail
of $\Y{}$ is exponentially decreasing.} for $c$ at or near the critical value
$\cS$.

This distressing fact only indicates that the normal and quasi{\zzz}normal
approximations are deficient at $\cS$. In contrast, the inverse Gaussian
approximation (see Fig.~\ref{sdtuyikghlhj}) does not have such a drawback. The
case when $c$ is close to $\cS$ or even equals $\cS$ is routinely covered by
the inverse Gaussian approximation.

\begin{remark}
The intuitive explanation of deficiency of normal and quasi{\zzz}normal
approximations in the vicinity of the critical value $\cS$ is given in
\cite{[Malinovskii=2021=a]}, Section~9.2.4. It is as follows. For exponentially
distributed $\X{i}$ and $\Y{i}$, $i=1,2,\dots$, and $c$ equal to $\cS$ the
compound sum is made up of the random variables with power moments of order no
greater than $1/2$. Therefore, there is no reason to expect validity of normal
and quasi{\zzz}normal approximations.\qed\end{remark}

The flaw of \eqref{asertgrfyhd} and \eqref{dfvdsfbdfbXX} in the vicinity od
$\cS$ is merely a drawback of a particular mathematical technique. However, it
has had a major negative impact on insurance modeling: while the insurance
process in a close neighborhood of the critical point $\cS$ is of great
importance (see details in \cite{[Malinovskii=2021=a]}), it was overlooked for
a long time partly due to this flaw. Moreover, the inverse Gaussian
approximation allows (see \cite{[Malinovskii=2021=c]}) to develop a new
approach to risk measures different from value at risk (VaR).

\section{Limit theorems for \bmRD\ and proper \bmID{\zzzx}sums}\label{saergsdyherd}

Kesten mentioned the part of modular analysis that deals with the lack of
independence. It is best known because the area of ``Markov chains, continuous
time processes, and other situations in which the independence between summands
no longer applies'' is of particular interest. Among the tricks which reduce
limit theorems for dependent summands to the case of independent summands,
Kesten mentioned ``D{\oe}blin's trick''.

Studying limit theorems for \mRD\ and proper \mID{\zzzx}sums, we will combine
``D{\oe}blin's trick'' and ``Blackwell's trick'' and extend it to compound sums
with general modular structure.

\subsection{Compound sums with general modular structure}

For $t>0$ and a complex basis $(\T{i},\X{i})$, $i=1,2,\dots$, we focus on
unaltered complex compound sum $\Sum{S}{\infLS{t}}$. We denote by
$0<\LRegt{0}<\LRegt{1}<\dots$ a sequence of random indices, put
\begin{equation*}
\begin{gathered}
\Block{\tau}{0}{}=\LRegt{0},\quad \Block{T}{0}{}=\sum_{i=1}^{\LRegt{0}}\;\T{i},
\quad\Block{X}{0}{}=\sum_{i=1}^{\LRegt{0}}\;\X{i},
\\[0pt]
\Block{\tau}{k}{}=\LRegt{k}-\LRegt{k-1},\quad
\Block{T}{k}{}=\sum_{i=\LRegt{k-1}+1}^{\LRegt{k}}\T{i},\quad
\Block{X}{k}{}=\sum_{i=\LRegt{k-1}+1}^{\LRegt{k}}\X{i},
\end{gathered}
\end{equation*}
where $k=1,2,\dots$, and introduce the following definition.

\begin{definition}\label{wryrtutyity}
We say that the basis $(\T{i},\X{i})$, $i=1,2,\dots$, allows a modular
structure if there exist \as finite random indices
$0<\LRegt{0}<\LRegt{1}<\dots$, such that
\begin{list}{}{\parsep=2pt\leftmargin=44pt\labelwidth=24pt\labelsep=8pt\topsep=4pt}
\item[\rmCase{i}] the random vectors
$(\Block{\tau}{k}{},\Block{T}{k}{},\Block{X}{k}{})$ are independent for
$k=0,1,\dots$, and identically distributed for $k=1,2,\dots$,
\item[\rmCase{ii}] $\Sum{V}{\LRegt{k}}>\max_{0\leqslant
i\leqslant\LRegt{k}-1}\Sum{V}{i}$, $k=0,1,\dots$, i.e., $\LRegt{i}$,
$i=0,1,\dots$, are upper record moments of the sequence
$\Sum{V}{n}=\sum_{i=1}^{n}\T{i}$, $n=1,2,\dots$.
\end{list}
\end{definition}

For modular \iid random vectors
$(\Block{X}{k}{},\Block{T}{k}{},\Block{\tau}{k}{})
\eqD(\Block{X}{}{},\Block{T}{}{},\Block{\tau}{}{})$, $k=1,2,\dots$, we use the
shorthand notation $\parMu{\Block{\tau}{}{}}{}=\E\Block{\tau}{}{}$,
$\parMu{\Block{T}{}{}}{}=\E\Block{T}{}{}$,
$\parMu{\Block{X}{}{}}{}=\E\Block{X}{}{}$,
$\parSigma{\Block{\tau}{}{}}{2}=\D\Block{\tau}{}{}$,
$\parSigma{\Block{T}{}{}}{2}=\D\Block{T}{}{}$,
$\parSigma{\Block{X}{}{}}{2}=\D\Block{X}{}{}$. For
$\centrBlock{\tau}{}{}=\Block{\tau}{}{}-\parMu{\Block{\tau}{}{}}{}$,
$\centrBlock{T}{}{}=\Block{T}{}{}-\parMu{\Block{T}{}{}}{}$,
$\centrBlock{X}{}{}=\Block{X}{}{}-\parMu{\Block{X}{}{}}{}$, and write
$\kov{\Block{T}{}{}\Block{\tau}{}{}}{}=\E\centrBlock{T}{}{}\!\centrBlock{\tau}{}{}$,
$\kov{\Block{X}{}{}\Block{T}{}{}}{}=\E\centrBlock{X}{}{}\centrBlock{T}{}{}$,
$\kov{\Block{X}{}{}\Block{\tau}{}{}}{}=\E\centrBlock{X}{}{}\centrBlock{\tau}{}{}$,
$\kor{\Block{T}{}{}\Block{\tau}{}{}}{}=\kov{\Block{T}{}{}\Block{\tau}{}{}}{}/
(\parSigma{\Block{T}{}{}}{}\parSigma{\Block{\tau}{}{}}{})$
$\kor{\Block{X}{}{}\Block{T}{}{}}{}=\kov{\Block{X}{}{}\Block{T}{}{}}{}/
(\parSigma{\Block{X}{}{}}{}\parSigma{\Block{T}{}{}}{})$
$\kor{\Block{X}{}{}\Block{\tau}{}{}}{}=\kov{\Block{X}{}{}\Block{\tau}{}{}}{}/
(\parSigma{\Block{X}{}{}}{}\parSigma{\Block{\tau}{}{}}{})$. We further
introduce zero{\zzz}mean block random variables\footnote{Note that
$\E\big(\Block{\Renyi{}}{\!\Block{X}{}{}\!\Block{T}{}{}}{2}\big)$ is the same
as $\D\big(\Block{\Renyi{}}{\!\Block{X}{}{}\!\Block{T}{}{}}{}\big)$ and
$\E\big(\Block{\Renyi{}}{\!\Block{\tau}{}{}\!\Block{T}{}{}}{2}\big)$ is the
same as $\D\big(\Block{\Renyi{}}{\!\Block{\tau}{}{}\!\Block{T}{}{}}{}\big)$.}
\begin{equation*}
\begin{aligned}
\Block{\Renyi{}}{\!\Block{T}{}{}\Block{\tau}{}{}}{}
&=\parMu{\Block{T}{}{}}{}\,\Block{\tau}{}{}-\parMu{\Block{\tau}{}{}}{}\Block{T}{}{}
&&& \Block{\Renyi{}}{\!\Block{X}{}{}\!\Block{T}{}{}}{}
=\parMu{\Block{T}{}{}}{}\Block{X}{}{}-\parMu{\Block{X}{}{}}{}\Block{T}{}{}
\\[-9pt]
&&\quad\text{and}\quad&&&
\\[-9pt]
&=\parMu{\Block{T}{}{}}{}\centrBlock{\tau}{}{}-\parMu{\Block{\tau}{}{}}{}\centrBlock{T}{}{},
&&&
=\parMu{\Block{T}{}{}}{}\centrBlock{X}{}{}-\parMu{\Block{X}{}{}}{}\centrBlock{T}{}{}
\end{aligned}
\end{equation*}
and put
\begin{equation*}
\begin{gathered}
\bMeanN{}{\!}=\parMu{\Block{\tau}{}{}}{}\parMu{\Block{T}{}{}}{-1}, \quad
\bStDevN{2}{}=\E\big(\Block{\Renyi{}}{\!\Block{T}{}{}\Block{\tau}{}{}}{2}\big)\,
\parMu{\Block{T}{}{}}{-3},
\\[2pt]
\bMeanS{}{\!}=\parMu{\Block{X}{}{}}{}\parMu{\Block{T}{}{}}{-1},\quad
\bStDevS{2}{}=\E\big(\Block{\Renyi{}}{\!\Block{X}{}{}\!\Block{T}{}{}}{2}\big)\parMu{\Block{T}{}{}}{-3},
\\[0pt]
\bKovSN{}{\!}=\E\big(\Block{\Renyi{}}{\!\Block{T}{}{}\Block{\tau}{}{}}{}
\Block{\Renyi{}}{\!\Block{X}{}{}\!\Block{T}{}{}}{}\big)\,\parMu{\Block{T}{}{}}{-3},
\quad\bFKovSN{}{\!}=\bKovSN{}{}\big(\bStDevS{}{}\bStDevN{}{\!}\big)^{-1}.
\end{gathered}
\end{equation*}

Assumption (i) tackles ``dependence complexity'', while (ii) takes hold of
``irregularity complexity'', in both cases by reducing the original complex
compound sum to a simple modular compound sum with simple modular basis
$(\Block{T}{k}{},\Block{X}{k}{})$, $k=1,2,\dots$. It consists of \iid modular
random vectors with $s${\zzz}components positive. Moreover, if
$\Sum{V}{\LRegt{k}}<t$, then crossing level $t$ cannot occur inside $k$th block
or inside blocks with a smaller index.

With $\Polynom{R}{}(x,y)=\dfrac{x-\bFKovSN{}{\!}y}{1-\bFKovSN{2}{}}$, we will
use the following notation:
\begin{equation*}
\bKompl{3}{0}=\,\E\big(\Block{\Renyi{}}{\!\Block{X}{}{}\!\Block{T}{}{}}{3}\big)
\parMu{\Block{T}{}{}}{-4}\bStDevS{-3}{}-3\Polynom{R}{}(\Value{L}{1},\Value{L}{2}),
\end{equation*}
where
\begin{equation*}
\begin{aligned}
\Value{L}{1}=&\,\big(\kov{\Block{X}{}{}\Block{\tau}{}{}}{}(\parMu{\Block{\tau}{}{}}{}
\parSigma{\Block{T}{}{}}{2}
-\parMu{\Block{T}{}{}}{}\kov{\Block{T}{}{}\Block{\tau}{}{}}{})
-\parSigma{\Block{\tau}{}{}}{2}
(\parMu{\Block{X}{}{}}{}\parSigma{\Block{T}{}{}}{2}-\parMu{\Block{T}{}{}}{}
\kov{\Block{X}{}{}\Block{T}{}{}}{})
\\[2pt]
&+\kov{\Block{T}{}{}\Block{\tau}{}{}}{}(\parMu{\Block{X}{}{}}{}
\kov{\Block{T}{}{}\Block{\tau}{}{}}{}-\parMu{\Block{\tau}{}{}}{}
\kov{\Block{X}{}{}\Block{T}{}{}}{})\big)
\parMu{\Block{T}{}{}}{-3}\bStDevS{-1}{}\bStDevN{-2}{},
\\[4pt]
\Value{L}{2}=&\,\big(-\kov{\Block{X}{}{}\Block{\tau}{}{}}{}(\parMu{\Block{X}{}{}}{}
\parSigma{\Block{T}{}{}}{2}
-\parMu{\Block{T}{}{}}{}\kov{\Block{X}{}{}\Block{T}{}{}}{})+\parSigma{\Block{X}{}{}}{2}
(\parMu{\Block{\tau}{}{}}{}\parSigma{\Block{T}{}{}}{2}-\parMu{\Block{T}{}{}}{}
\kov{\Block{T}{}{}\Block{\tau}{}{}}{})
\\[2pt]
&-\kov{\Block{X}{}{}\Block{T}{}{}}{}(\parMu{\Block{\tau}{}{}}{}\kov{\Block{X}{}{}\Block{T}{}{}}{}
-\parMu{\Block{X}{}{}}{}\kov{\Block{T}{}{}\Block{\tau}{}{}}{})\big)
\parMu{\Block{T}{}{}}{-3}\bStDevS{-2}{}\bStDevN{-1}{}.
\end{aligned}
\end{equation*}
It can be shown by direct algebra that
$\Polynom{R}{}\big(\Value{L}{1},\Value{L}{2}\big)=-\big(\parMu{\Block{X}{}{}}{}
\parSigma{\Block{T}{}{}}{2}
-\parMu{\Block{T}{}{}}{}\kov{\Block{X}{}{}\Block{T}{}{}}{}\big)\,\bStDevS{-1}{}\parMu{\Block{T}{}{}}{-2}$.
Finally, we introduce
\begin{equation*}
\begin{aligned}
\xEta{1}&=\E\Block{X}{0}{}+\parMu{\Block{T}{}{}}{-1}
\sum_{s=1}^{\infty}\int_{-\infty}^{\infty}
x\int_{0}^{\infty}\P\big\{\Block{\tau}{1}{}\geqslant
s,\Block{N}{}{}(t)=s,\Block{X}{0}{\hskip 1pt(s)}\in dx\big\}\,dt,
\\[0pt]
\xEta{2}&=\E\Block{T}{0}{}+\parMu{\Block{T}{}{}}{-1}
\sum_{s=1}^{\infty}\int_{0}^{\infty}t\,\P\big\{\Block{\tau}{1}{}\geqslant
s,\Block{N}{}{}(t)=s\big\}\,dt,
\\[0pt]
\xEta{3}&=\E\Block{\tau}{0}{}+\parMu{\Block{T}{}{}}{-1}
\sum_{s=1}^{\infty}s\int_{0}^{\infty}\P\big\{\Block{\tau}{1}{}\geqslant
s,\Block{N}{}{}(t)=s\big\}\,dt,
\end{aligned}
\end{equation*}
where $\Block{X}{0}{\hskip 1pt(s)}=\sum_{k=\LRegt{0}+1}^{\LRegt{0}+s}\X{k}$ and
$\Block{N}{}{}(t)=\inf\big\{n\geqslant
1:\sum_{k=\LRegt{0}+1}^{\LRegt{0}+n}\T{k}>t\big\}$, or $\infty$, if the set is
empty.

Theorem~\ref{srfgdhftgdth} below is formulated under the following conditions.

\begin{condition}{\KondIbf{V}}\,(Non{\zzz}degenerate correlation matrix)\label{srdtgrehrt}
$\Det\McoR{Q}{}>0$ for the correlation matrix
\begin{equation*}
\McoR{Q}{}=\KRm{\Block{X}{1}{}}{\Block{T}{1}{}}{\Block{\tau}{1}{}}
=\left(\begin{matrix} 1 & \kor{\Block{X}{}{}\Block{T}{}{}}{} &
\kor{\Block{X}{}{}\Block{\tau}{}{}}{}
\\[0pt]
\kor{\Block{X}{}{}\Block{T}{}{}}{} & 1 & \kor{\Block{T}{}{}\Block{\tau}{}{}}{}
\\[0pt]
\kor{\Block{X}{}{}\Block{\tau}{}{}}{} & \kor{\Block{T}{}{}\Block{\tau}{}{}}{} &
1
\end{matrix}\right).
\end{equation*}
\end{condition}

\begin{condition}{\KondIIbf{P}{}}\,(Bounded density condition)\label{sdrtyrtute}
For an integer $n\geqslant 1$, there exists a bounded convolution power
$\pdF{\Block{T}{}{},\Block{\tau}{}{}}^{*n}(x,n)$ \wrt Lebesgue measure.
\end{condition}

\begin{condition}{\KondIbf{L}}\,(Lattice condition)\label{asrdtgrehrtr}
The integer{\zzz}valued random variable $\Block{\tau}{1}{}$ assumes values in
the set of natural numbers with maximal span $1$.
\end{condition}

\begin{condition}{\KondIbf{C}}\,(Uniform Cram{\'e}r's condition)\label{asrtyr6ui6}
\begin{equation*}
\overline{\lim}_{|t_{1}|\to\infty}\big|\,\E\exp\big\{it_{1}\Block{X}{}{}
+it_{2}\Block{T}{}{}+it_{3}\Block{\tau}{}{}\big\}\big|<1
\end{equation*}
for all $t_{2},t_{3}\in\Rline$.
\end{condition}

\begin{condition}{\KondIIIbf{B}{r}{s}}\,(Block moments
condition)\label{srdthrewr} For $r,s>0$
\begin{list}{}{\topsep=8pt\parsep=-2pt\leftmargin=70pt\labelwidth=24pt\labelsep=8pt}
\item[\rmCase{i}] $\E\hskip 1pt\big(\LRegt{0}^{\hskip 1pt r}\big)<\infty$,\quad
$\E\big(\Block{\tau}{1}{s}\big)<\infty$,
\item[\rmCase{ii}]
$\E\bigg(\bigg[\displaystyle\sum_{k=1}^{\LRegt{0}}\big|\,\T{k}\big|\,\bigg]^{r}\bigg)<\infty$,\quad
$\E\bigg(\bigg[\displaystyle\sum_{k=\LRegt{0}+1}^{\LRegt{1}}\big|\,\T{k}\big|\,\bigg]^{\hskip
-1pt s}\bigg)<\infty$,
\item[\rmCase{iii}]
$\E\bigg(\bigg[\displaystyle\sum_{k=1}^{\LRegt{0}}\big|\,\X{k}\big|\,\bigg]^{r}\bigg)<\infty$,\quad
$\E\bigg(\bigg[\displaystyle\sum_{k=\LRegt{0}+1}^{\LRegt{1}}\big|\,\X{k}\big|\,\bigg]^{\hskip
-1pt s}\,\bigg)<\infty$.
\end{list}\vskip -6pt
\end{condition}

The following theorem is Theorem~1\,(III) in \cite{[Malinovskii=1994=a]}.

\begin{theorem}\label{srfgdhftgdth}
Assume that the basis $(\T{i},\X{i})$, $i=1,2,\dots$, of the complex compound
sum $\Sum{S}{\infLS{t}}$ allows a modular structure. If conditions \KondIit{V},
\KondIIit{P}{}, \KondIit{L}, \KondIit{C}, and \KondIIIit{B}{k-\frac{3}{2}}{k}
with $k>3$ are satisfied, then there exist polynomials $\Polynom{Q}{r}(x)$,
$r=1,2,\dots,k-3$, of degree $3r-1$ such that
\begin{equation*}
\begin{aligned}
&\sup_{x\in\Rline}\,\bigg|\,\P\big\{\Sum{S}{\infLS{t}}-\bMeanS{}{t}\leqslant
x\bStDevS{}{t^{1/2}}\big\}-\UGauss{0}{1}(x)
\\[0pt]
&\hskip 90pt-\sum_{r=1}^{k-3}
\frac{\Polynom{Q}{r}(x)}{t^{r/2}}\,\Ugauss{0}{1}(x)\,\bigg|=\simb{O}\big(t^{(k-2)/2}\big),\quad
t\to\infty.
\end{aligned}
\end{equation*}
In particular,
\begin{equation*}
\Polynom{Q}{1}(x)=-\frac{1}{6}\;\bKompl{3}{0}\Herm{2}(x)+\Polynom{R}{}(\Value{L}{1},\Value{L}{2})-
\frac{\parMu{\Block{T}{}{}}{}\xEta{1}-\parMu{\Block{X}{}{}}{}\xEta{2}}{\bStDevS{}{}\parMu{\Block{T}{}{}}{}}.
\end{equation*}
\end{theorem}

\begin{proof}[Sketch of proof of Theorem~\ref{srfgdhftgdth}]
The proof, similarly to the proof of Theorem~\ref{drsthtjmkfh}, consists of
steps~A--F.

\Step{A}{use of fundamental identity}
Similar to \eqref{srdtuytfgik}, we use the total probability rule. Writing
\begin{equation*}
\begin{gathered}
\Block{X}{}{[r]}=\sum_{i=1}^{r}\X{i},\ \Block{T}{}{[r]}=\sum_{i=1}^{r}\T{i},\
\Block{X}{}{(s)}=\sum_{i=\LRegt{m}+1}^{\LRegt{m}+s}\X{i},\
\Block{T}{}{(s)}=\sum_{i=\LRegt{m}+1}^{\LRegt{m}+s}\T{i}
\end{gathered}
\end{equation*}
and $\Block{N}{}{}(t)=\inf\big\{s\geqslant 1:\Block{T}{}{(s)}>t\big\}$, we have
\begin{equation*}
\begin{aligned}
&\P\big\{\Sum{S}{\infLS{t}}\leqslant x\big\}
=\,\sum_{n=1}^{\infty}\P\big\{\Sum{S}{\infLS{t}}\leqslant x,\infLS{t}=n\big\},
\end{aligned}
\end{equation*}
where the right{\zzz}hand side is the sum of, first, the expression
\begin{equation*}
\begin{aligned}
&\sum_{n=1}^{\infty}\sum_{r}\iint_{\substack{0\leqslant t_{1}\leqslant t\\[1pt] 0\leqslant x_{1}\leqslant x}}
\P\Big\{\Block{\tau}{0}{}=r,\Block{T}{}{[r]}\in dt_{1},\Block{X}{}{[r]}\in
dx_{1}\Big\}
\\[0pt]
&\hskip 30pt\times\P\Big\{\Block{\tau}{1}{}\geqslant
n-r,\Block{N}{}{}(t-t_{1})=n-r,\Block{X}{}{(n-r)}\leqslant x-x_{1}\Big\},
\end{aligned}
\end{equation*}
which corresponds to the absence of complete blocks, and, second, the
expression
\begin{equation}\label{asdfgsdhbrt}
\begin{aligned}
&\sum_{n=1}^{\infty}\sum_{r,s}\sum_{m=1}^{n-(r+s)}
\idotsint_{\substack{0\leqslant t_{1}+t_{2}\leqslant t\\[1pt] 0\leqslant x_{1}+x_{2}\leqslant x}}
\P\bigg\{\sum_{i=1}^{m}\Block{X}{i}{}\leqslant x-\big(x_{1}+x_{2}\big),
\\[0pt]
&\hskip 100pt\sum_{i=1}^{m}\Block{\tau}{i}{}=n-(s+r)\,
\bigg|\,\sum_{i=1}^{m}\Block{T}{i}{}=t-(t_{1}+t_{2})\big)\bigg\}
\\[0pt]
&\hskip
30pt\times\pdF{\Block{T}{}{}}^{*m}\big(t-(t_{1}+t_{2})\big)\,\P\Big\{\Block{\tau}{0}{}=r,\Block{T}{}{[r]}\in
dt_{1},\Block{X}{}{[r]}\in dx_{1}\Big\}
\\[0pt]
&\hskip 30pt\times\P\Big\{\Block{\tau}{1}{}\geqslant
s,\Block{N}{}{}(t_{2})=s,\Block{X}{}{(s)}\in dx_{2}\Big\}\,dt_{2}.
\end{aligned}
\end{equation}
It is clear that for large $t$ only the expression \eqref{asdfgsdhbrt} matters.

In words, in \eqref{asdfgsdhbrt} the first incomplete block is made up by $r$
summands and the value $t$ cannot be exceeded all over this block because
$t_{1}<t$; there are $m$ complete blocks, and exceeding the value $t$ still
cannot happen, but it does occurs on the $s$th summand of the the $(m+1)$th
complete block consisting of more than $s$ summands.

\Steps{B and C}
remain essentially the same as in the proof of Theorem~\ref{drsthtjmkfh}.

\Step{D}{application of refined \cCLT}
The integrand in \eqref{asdfgsdhbrt} is ready for application of Edgeworth's
expansion in three{\zzz}dimensional hybrid integro{\zzz}local{\zzz}local \CLT
with non{\zzz}uniform remainder term. These theorems see in
\cite{[Dubinskaite=1982]}--\cite{[Dubinskaite=1984=b]}.

\Steps{E and F}
are technically much more complicated than in the proof of refined elementary
renewal theorem \eqref{we567u65i} and Theorem~\ref{drsthtjmkfh} (they require a
number of special identities), but use essentially the same analytical methods.
\end{proof}

\begin{remark}
Theorem~\ref{drsthtjmkfh} is a corollary of Theorem~\ref{srfgdhftgdth}. For
\mII{\zzzx}basis $(\T{i},\X{i})$, $i=1,2,\dots$, the modular structure is
yielded by random indices $0\equiv\Ladd{0}<\Ladd{1}<\dots$ and, using Wald's
identities, $\bMeanS{}{}$ is transformed into $\MeanS{}{}$, and $\bKompl{3}{0}$
is transformed into $\Kompl{3}{0}$.\qed\end{remark}

\subsection{Limit theorems for \bmRD{\zzzx}sums with Markov dependence}\label{srtgrghsd}

Markov renewal process (see, e.g., \cite{[Cinlar=1975=b]},
\cite{[Pacheco=Prabhu=Tang=2009]}) is a homogeneous two{\zzz}dimensional Markov
chain $(\cX{k},\T{k})$, $k=0,1,\dots$, which takes values in a general state
space $(\Space{E}\times\PRline,\Field{E}\otimes\Field{R})$. Its transition
functions are defined by a semi{\zzz}Markov transition kernel. Markov renewal
processes and their counterpart, semi{\zzz}Markov processes, are among the most
popular models of applied probability\footnote{See, e.g., two bibliographies on
semi{\zzz}Markov processes \cite{[Teugels=1976]} and \cite{[Teugels=1986]}. The
former consists of about 600 papers by some 300 authors and the later of almost
a thousand papers by more than 800 authors.}.

In the same way as for renewal process, this model can be reformulated in terms
of \mRD{\zzzx}basis $(\T{i},\X{i})$, $i=1,2,\dots$, and \mRD{\zzzx}sum with
Markov dependence. Under certain regularity conditions, the sequence
$0<\LRegt{0}<\LRegt{1}<\dots$ from Definition~\ref{wryrtutyity} is built in
\cite{[Athreya=McDonald=Ney=1978]}, \cite{[Nummelin=1978=b]}. Theorem analogous
to Theorem~\ref{srfgdhftgdth} see in \cite{[Malinovskii=1991]}.

\begin{remark}[Refined \CLT for Markov chains]
In the special case $\T{i}\equiv 1$, the \mRD{\zzzx}sum with Markov dependence
becomes the ordinary sum $\Sum{S}{n}$ of random variables defined on the Markov
chain $\{\cX{i}\}_{i\geqslant 0}$.

Originally, ``D{\oe}blin's trick'' was applied to $\Sum{S}{n}$ with discrete
$\{\cX{i}\}_{i\geqslant 0}$. Subsequently, it was repeatedly noted (see, e.g.,
\cite{[Dacunha=Castelle=Duflo=1986]} and \cite{[Meyn=Tweedie=2009]}) that not
the full countable space structure is often needed, but just the existence of
one single ``proper'' point. The results then carry over with only notational
changes to the countable case and, moreover, to (by means of the ``splitting
technique'' and its counterparts, see \cite{[Athreya=Ney=1978]},
\cite{[Nummelin=1978=b]}, \cite{[Nummelin=1984]}, and
\cite{[Meyn=Tweedie=2009]}) irreducible recurrent general state space Markov
chain $\{\cX{i}\}_{i\geqslant 0}$ which satisfies strong minorization
condition. Using modular advanced technique, for discrete (see
\cite{[Bolthausen=1980]})  and general state space Markov chains (see
\cite{[Bolthausen=1982=a]}, \cite{[Malinovskii=1987=b]},
\cite{[Malinovskii=1990]}) Berry{\zzz}Esseen's estimate, Edgeworth's
expansions, and asymptotic results for large deviations in \CLT have been
obtained.\qed\end{remark}

\subsection{Limit theorems for proper \bmID{\zzzx}sums with Markov
dependence}\label{srtgdfhrf}

Markov additive process (see, e.g., \cite{[Cinlar=1972]},
\cite{[Pacheco=Prabhu=Tang=2009]}) is a generalization of Markov renewal
process with $(\cX{k},\T{k})$, $k=0,1,\dots$, taking values in
$(\Space{E}\times\Rline,\Field{E}\otimes\Field{R})$ rather than in
$(\Space{E}\times\PRline,\Field{E}\otimes\Field{R})$. Under certain regularity
conditions, for proper \mID{\zzzx}sums with Markov dependence, the sequence
$0<\LRegt{0}<\LRegt{1}<\dots$ from Definition~\ref{wryrtutyity} is built in
\cite{[Alsmeyer=2000]} and \cite{[Alsmeyer=2018]}. Loosely speaking, these are
the ladder points at which regeneration occurs.

\section{Analysis of defective \bmID{\zzzx}sums}\label{srfydfuhjfg}

For defective \mID{\zzzx}sums, a deep analytical insight, such as in the case
of defective \mII{\zzzx}sums, is (see, e.g.,
\cite{[Mikosch=Samorodnitsky=2000=b]},
\cite{[Nyrhinen=1998]}--\cite{[Nyrhinen=2001]}) hardly possible. An alternative
is computer{\zzz}intensive analysis (see, e.g., \cite{[Albrecher=Kantor=2002]},
\cite{[Lehtonen=Nyrhinen=1992=b]}).

Note that the papers devoted to defective \mID{\zzzx}sums with Markov
dependence are formulated as a study of the ruin probabilities for risk
processes of Markovian type. This link with the ruin theory not only suggests
rational examples of defective \mID{\zzzx}sums, but also indicates areas in
which they are of major interest.

\section{A concluding remark}\label{rgryhdhf}

Due credits must be given to D{\oe}blin (1915{\zzz}1940), who (see
\cite{[Levy=1955]}, \cite{[Lindvall=1991]}, \cite{[Meyn=Tweedie=1993]}) put
forth (see \cite{[Deblin=1938=a]}--\cite{[Deblin=Fortet=1937=a]}) the major
innovative ideas of coupling, decomposition and dissection\footnote{The English
translation of the title of \cite{[Deblin=1938=a]} is ``On two problems of
Kolmogorov concerning countable Markov chains'', with reference to
\cite{[Kolmogorov=1936]}. Therefore, D{\oe}blin's glory in this regard is
partly shared with Kolmogorov.} in Markov chains. The contributions of many
scientists after him are mere a development of his seminal ideas. In this
regard, compelling is the phrase from \cite{[Griffeaths=1978]}, that ``our goal
is to isolate the key ingredients in D{\oe}blin's proof and use them to derive
ergodic theorems for much more general Markov processes on a typically
uncountable state space''.

Having created a tool for sequential analysis, Anscombe (see
\cite{[Anscombe=1952]}) rediscovered that part of ``D{\oe}blin's trick'' (as
Kesten called it in \cite{[Kesten=1977]}), or ``D{\oe}blin's method'' (as
Kolmogorov called it in \cite{[Kolmogorov=1949]}), which is the use of
Kolmogorov's inequality for maximum of partial sums. R{\'e}nyi (see
\cite{[Renyi=1957]}, \cite{[Renyi=1960]}), when he adapted Anscombe's advance
to compound sums $\Sum{S}{\infLS{t}}$ with $\infLS{t}$ more general\footnote{In
R{\'e}nyi's sums (see Remark~\ref{srgtfdhnfdhn}), $\infLS{t}$ is dependent on
\iid summands and $\infLS{t}/\,t\toP\varrho$, $t\to\infty$, but is of a general
form.} than in \eqref{xfghmhrey}, acknowledged (see \cite{[Renyi=1960]}) that
``K.L.\,Chung kindly called my attention to the fact that the main idea of the
proof (the application of the inequality of Kolmogorov $\dots$) given in
\cite{[Renyi=1957]} is due to D{\oe}blin (see \cite{[Deblin=1938=a]}). It has
been recently proved $\dots$ that the method in question can lead to the proof
of the most general form of Anscombe's theorem.'' Just the way it is in the
world of ideas.


\end{document}